\documentclass[11pt,thmsa]{amsart}

\usepackage[all]{xy}
\usepackage{color} 
\usepackage{hyperref} 
\usepackage{latexsym, amsmath, amssymb, amsfonts, amscd,bm}
\usepackage{amsthm}
\usepackage{t1enc}
\usepackage[mathscr]{eucal}
\usepackage{indentfirst}
\usepackage{graphicx, pb-diagram}
\usepackage{fancyhdr}
\usepackage{fancybox}

 \newtheorem{theorem}{Theorem}[section]
 \newtheorem{lemma}[theorem]{Lemma}
 \newtheorem{proposition}[theorem]{Proposition}
 
 \newtheorem{corollary}[theorem]{Corollary}
 \newtheorem{definition}[theorem]{Definition}
  \theoremstyle{definition}
 \newtheorem{example}[theorem]{Example}
 \newtheorem{remark}[theorem]{Remark}

\renewcommand{\graph}{\mathrm{graph}}

\newcommand{\cC}{\mathcal{C}}

\newcommand{\ZZ}{\ensuremath{\mathbb Z}}

\newcommand{\RR}{\ensuremath{\mathbb R}}
\newcommand{\TT}{\ensuremath{\mathbb T}}

\newcommand{\pd}[1]{\frac{\partial}{\partial #1}} 
\newcommand{\pr}{\mathrm{p}^{\mathrm{v}}}

\newcommand{\Defor}{\mathsf{Def}}
\newcommand{\Ham}{\mathsf{Ham}}
\newcommand{\Sym}{\mathsf{Sym}}
\newcommand{\Pois}{\mathsf{Pois}}
\newcommand{\Moduli}{\mathcal{M}}

\newcommand{\gauge}{\mathsf{gauge}}
\newcommand{\ex}{\mathsf{ext}}

\newcommand{\para}{\kern0.2em{\backslash} \kern-0.7em {\backslash} \kern0.2em }
\newcommand{\parap}{          {\backslash} \kern-0.7em {\backslash} \kern0.2em }

\begin{document}

 \title{Equivalences of coisotropic submanifolds}

\author{Florian Sch\"atz}
\email{florian.schaetz@gmail.com}
\address{Centre for Quantum Geometry of Moduli Spaces, Aarhus University,
Ny Munkegade 118, DK-8000 Aarhus C, Denmark.
}
 
 \author{Marco Zambon}
\email{marco.zambon@uam.es, marco.zambon@icmat.es, marco.zambon@wis.kuleuven.be}
\address{Universidad Aut\'onoma de Madrid (Departamento de Matem\'aticas) and ICMAT(CSIC-UAM-UC3M-UCM),
Campus de Cantoblanco,
28049 - Madrid, Spain. Current address: KU Leuven, Department of Mathematics, Celestijnenlaan 200B box 2400, BE-3001 Leuven, Belgium.}


 \begin{abstract}
We study the role that Hamiltonian and symplectic diffeomorphisms
play in the deformation problem of coisotropic submanifolds. We prove that the action by Hamiltonian diffeomorphisms
corresponds to the gauge-action of the $L_\infty$-algebra of Oh and Park.
Moreover we introduce the notion of extended gauge-equivalence
and show that in the case of
 Oh and Park's $L_\infty$-algebra
one recovers the action of symplectic isotopies on coisotropic submanifolds.
Finally, we consider the transversally integrable case in detail.
\end{abstract}

\maketitle

\setcounter{tocdepth}{1} 
\tableofcontents

 \section*{Introduction}
 
Coisotropic submanifolds form an important class of sub-objects in symplectic and Poisson geometry.
They naturally generalize Lagrangian submanifolds,
play an important role in the theory of constraints and also
appear in theoretical physics in the form of ``branes'', i.e. boundary conditions of
sigma models \cite{KapustinOrlov,CattaneoFelder}.

In this note we consider coisotropic deformations inside a symplectic manifold.
The nearby deformations of a Lagrangian submanifold $L$ are well-understood:
by Weinstein's normal form theorem, one can replace the ambient symplectic
manifold by the cotangent bundle $T^*L$. The graph of a $1$-form $\alpha$
is Lagrangian if and only if $\alpha$ is closed.
If one identifies closed $1$-forms which are related through an Hamiltonian isotopy,
one arrives at the first de Rham cohomology group $H^1(L,\RR)$ of $L$ as
the appropriate moduli space of nearby Lagrangian deformations.

The generalization of these statements to coisotropic submanifolds is not obvious, since
the space of coisotropic deformations is not linear and not even modelled on a topological vector space,
see \cite{coisoemb,OP}.
However, the general pattern of deformation theory teaches us that every deformation problem\footnote{... in characteristic zero...} should be captured by differential graded Lie algebras or their homotopical cousins, known as $L_\infty$-algebras.
That this is indeed the case was established by Oh and Park in \cite{OP}.
To be more precise, Oh and Park constructed an $L_\infty$-algebra that controls the formal deformation
problem for coisotropic submanifolds. In the special case of a Lagrangian submanifold $L$,
their construction recovers the de Rham complex of $L$.

In \cite{OPPois}, we studied convergence issues arising in the framework of \cite{OP}.
One finds that 
the {\em Maurer-Cartan equation}, which replaces the condition of being closed from the Lagrangian case,
is always convergent, and  that it converges to zero if and only if one is dealing with a coisotropic deformation.\footnote{For an alternative
treatment of the coisotropic deformation problem in terms of a Maurer-Cartan equation, see \cite{FlorianBFV}.}

Having established a firmer link to actual geometric deformations, it is natural to turn attention to the geometric symmetries
that are present in the problem. In particular, one might wonder how the actions of Hamiltonian and symplectic isotopies on the space
of coisotropic deformations can be understood. A natural symmetry acting on  Maurer-Cartan elements of  Oh and Park's $L_\infty$-algebra
are the inner
automorphisms, known as {\em gauge-transformations}. Our main result is  that these agree with the action by Hamiltonian isotopies, while 
the action by symplectic isotopies agrees with certain extended gauge-equivalences, which we specify below.\\

In Section \ref{section: Ham diffeos} we deal with Hamiltonian isotopies. It turns out that the gauge-transformations of Oh and Park's $L_\infty$-algebra
correspond to certain special Hamiltonian isotopies. The remaining problem is to show that any Hamiltonian isotopy
can be reduced to such a special one. This is parallel to the Lagrangian situation: there the main task
is also to show that an arbitrary Hamiltonian isotopy can be reduced to a function $f$ on the Lagrangian submanifold, which acts on the space of closed $1$-forms (whose graphs we are interested in) simply by $\alpha \mapsto \alpha + df$.
We establish the appropriate generalization in Theorem \ref{cor:3agree}, Subsection \ref{subsection: Ham diffeos - equivalences match}.
As a consequence, we identify
$$\frac{\{\text{coisotropic submanifolds}\}}{\text{Hamiltonian isotopies}}\cong
 \frac{\{\text{Maurer-Cartan elements}\}}{\text{gauge-equivalences}},$$
which is the content of Theorem \ref{theorem: Hamiltonian equivalence}.
For an alternative treatment within the BFV-formalism we refer to the article \cite{FlorianModuli} by
the first named author.

Section \ref{section: symplectomorphisms} is concerned with symplectic isotopies.
Given a Lagrangian submanifold, any of its Lagrangian deformations is related to the original submanifold
by a symplectomorphism, so we do not obtain an interesting moduli space. In the general coisotropic case
the situation is much more complicated and we do obtain another reasonable equivalence relation
on the space of deformations by considering symplectic isotopies.
In order to fit this into the algebraic framework, 
we review the construction of Oh and Park's $L_\infty$-algebra \cite{OP}\cite{CaFeCo2}
using  Voronov's derived bracket construction \cite{vor1,vor2}.

We show that every $L_\infty$-algebra which arises through Voronov's construction comes along
with additional automorphisms. As a consequence, we obtain more ways to identify Maurer-Cartan elements.
We refer to this extended equivalence relation as {\em extended gauge-equivalence}.
The content of Theorem \ref{cor:3agree - sym}, Subsection \ref{subsection: symplectomorphisms - equivalences coincide} is that if one applies this {construction} to Oh and Park's $L_\infty$-algebra,
one precisely recovers the action of symplectic isotopies on the space of coisotropic deformations.
As a consequence, we {obtain the identification}
$$\frac{\{\text{coisotropic submanifolds}\}}{\text{symplectic isotopies}}\cong
 \frac{\{\text{Maurer-Cartan elements}\}}{\text{extended gauge-equivalences}},$$
see Theorem \ref{theorem: symplectic equivalence}.

In Section \ref{section: transversal}, we consider coisotropic submanifolds which are transversally integrable.
This regularity condition allows one to make some of the previous constructions more explicit. In particular,
one can give a formula for nearby coisotropic deformations which are obtained by an Hamiltonian
or symplectic isotopy from the original coisotropic submanifold, see Proposition \ref{prop:section}.

{In Appendix \ref{appendix: fibrewise entire} we discuss the extension of our results to fibrewise entire Poisson structures.
In \cite{OPPois} it was shown that the coisotropic deformation
problem for those Poisson structures is also controlled by an $L_\infty$-algebra. Most of the results established
in the bulk of the paper carry over to the case of fibrewise entire Poisson structures. We explain the necessary modifications
in the appendix.}
\\

\noindent{\bf Organization of the paper:}
 In Section \ref{section: pre-symplectic geometry}
we recall background material on coisotropic submanifolds.
In Section \ref{section: coisotropic submanifolds} we review the results
about deformations of coisotropic submanifolds which are relevant in the subsequent discussion.
In particular, we introduce Oh and Park's $L_\infty$-algebra and review
the relation between its Maurer-Cartan elements and the deformation problem.
In Section \ref{section: Ham diffeos} we discuss Hamiltonian isotopies,
while in Section \ref{section: symplectomorphisms} we deal with symplectic isotopies.
In Section \ref{section: transversal}, we consider the case of transversally integrable submanifolds.
{Finally, Appendix \ref{appendix: fibrewise entire} describes the extension of our results
to fibrewise entire Poisson structures.}
\\
 
\noindent{\bf Comparison with the literature:}
While we were completing this note, the preprint 
\cite{LOTVcoiso} by
L{\^e}, Oh, Tortorella and  Vitagliano
 appeared. It considers coisotropic deformations in the very general setting of abstract Jacobi manifolds, which include Poisson and symplectic manifolds as special cases.
 There is an overlap between the results presented there in \cite[Subsection 4.4]{LOTVcoiso} - once specialized to the symplectic case - and one of the main sections of the present note, namely Section 3.
In particular,  
 Thm. \ref{cor:3agree}
(i.e. the equivalence of 
Hamiltonian equivalence and  gauge-equivalence, under a compactness assumption) corresponds to  \cite[Corollary 4.24]{LOTVcoiso}. 
Notice that in the latter the assumption on the compactness of the coisotropic submanifold is omitted.\\

 \noindent{\bf Acknowledgements:}
M.Z. thanks Luca Vitagliano and Alberto Mart\'in Zamora for useful conversations. In particular, the proof of Lemma \ref{lem:techhard} was communicated to us by Luca Vitagliano. Moreover we thank H\^ong V\^an L\^e for useful comments on a draft-version of this note.
 Last but not least we thank the referee for the helpful suggestions which improved the manuscript.

 F.S. was supported by the Center of Excellence Grant ``Centre for Quantum Geometry of Moduli Spaces'' from the Danish National Research Foundation (DNRF95).
 M.Z. was partially supported by grants  
MTM2011-22612 and ICMAT Severo Ochoa  SEV-2011-0087 (Spain), and
Pesquisador Visitante Especial grant  88881.030367/2013-01 (CAPES/Brazil).
 
\section{(Pre-)Symplectic Geometry}\label{section: pre-symplectic geometry}

We summarize background information about coisotropic submanifolds and associated structures.

\begin{remark}
Throughout this paper, $(M,\omega)$ will denote a symplectic manifold.
 {Let $C$ be a submanifold of $M$ and $E\to C$ a vector subbundle of $TM\vert_C$.}
The symplectic orthogonal $E^\perp$ to $E$ is the vector bundle
whose fibre over $x\in C$ is
$$E^\perp_x := \{e\in T_xM \textrm{ such that } \forall v\in  {E_x} \textrm{ we have } \omega_x(e,v)=0\}.$$
Another way to characterize $E^\perp$ is as the pre-image of the annihilator $E^\circ$
of $E$ under the sharp-map
$$ \omega^{\sharp}: TM \to T^*M, \quad v\mapsto \omega(v,-).$$
\end{remark}
 
\begin{definition}\label{def:coisosection}
A submanifold $C$ of $(M,\omega)$ is {\bf coisotropic} if the symplectic orthogonal $TC^\perp$ to $TC$
is contained in $TC$.
\end{definition}

\begin{remark}
An alternative way to express the coisotropicity of $C$ 
is in terms of the Poisson bivector field $\Pi$ associated to $\omega$,
defined by the requirement that $\Pi^{\sharp}: T^*M \to TM, \quad \xi\mapsto \Pi(\xi,-)$ equals $-(\omega^{\sharp})^{-1}$.
Let $\mathcal{X}^{\bullet}(M)$ denote the space of multivector-fields
on $M$, i.e. sections of $ \wedge TM$.
There is a natural projection map
$$ P: \chi^{\bullet}(M) \to \Gamma(\wedge (TM\vert_C/TC)),$$
which is given by restricting  multivector-fields to $C$, followed by composition with the natural projection
 $\wedge TM\vert_C \to \wedge (TM\vert_C/TC)$.
The submanifold $C$ is coisotropic if and only if the Poisson bivector field $\Pi$
lies in the kernel of $P$.
\end{remark}

\begin{definition}
A two-form $\eta$ on $C$ that is closed and whose rank is constant
is called a {\bf pre-symplectic structure}.
\end{definition}

\begin{lemma}
Let $C$ be a coisotropic submanifold of $(M,\omega)$.
The pull back of $\omega$ to $C$ along the inclusion $\iota: C\hookrightarrow M$ is a closed two-form of constant
rank  $2\dim C - \dim M$. We denote   {this pre-symplectic structure} by $\omega_C$.
\end{lemma}

\begin{remark}\label{rem:isom}
\hspace{0cm}
\begin{enumerate}
 \item Let $\eta$ be any pre-symplectic structure on $C$. The closedness of $\eta$ implies that the kernel
 of $\eta^{\sharp}: TC \to T^*C$ is an involutive subbundle of $TC$. Hence $C$ is equipped with a foliation, called the characteristic foliation of
 $\eta$.
 \item We now consider the case of the pre-symplectic structure $\omega_C$ associated to a coisotropic submanifold
 $C$ of $(M,\omega)$. We always denote the kernel of $\omega_C$
 by $K (=TC^\perp)$ and the corresponding characteristic foliation by $\mathfrak{F}$
 in this situation.
 Moreover, observe that, in this situation, the vector bundle morphism
 $$
\xymatrix{
TM\vert_C \ar[r]^{\omega^{\sharp}} & T^*M\vert_C \ar[r] & (TC^\perp)^*
}
$$
is surjective and has kernel
$TC$.
Hence we obtain an isomorphism between the normal bundle $TM\vert_C / TC$ and $K^*$.
\end{enumerate}
\end{remark}

\begin{remark}
 We saw that every coisotropic submanifold comes along with a pre-symplectic structure.
An important observation is that this can be reversed: {\em every} pre-symplectic
structure can be realized as the pre-symplectic structure associated to a coisotropic submanifold. Moreover, this
realization is essentially unique.
We start with a pre-symplectic structure $\eta$ on a manifold $C$. Let $K$ be the kernel of $\eta^{\sharp}$ and $G$ a complement to $K$. The choice of $G$ yields an inclusion
$j: K^*\hookrightarrow T^*C$. Recall that $T^*C$ carries a canonical symplectic structure $\omega_{T^*C}$.
We now combine $\eta$ and $\omega_{T^*C}$ into the two-form
 $$\Omega:= \pi^*\omega_C + j^* \omega_{T^*C}.$$
on $K^*$, where $\pi$ denotes the projection map $K^*\to C$. 

The two-form $\Omega$ restricts to $\eta$ on $C$ and is symplectic
on a tubular neighborhood $U$ of the zero section $C \subset K^*$. We refer to $(U,\Omega)$  as the {\em local
symplectic model} associated to the
the pre-symplectic manifold $(C,\eta)$.

The local symplectic model depends on the choice of complement $G$ to $K$, but choosing different
complements will lead to local symplectic models which are symplectomorphic in neighborhoods of $C$, and one
can choose a symplectomorphism that restricts to the identity on $C$. Hence we will
speak of {\em the} local symplectic model of $(C,\eta)$.

The following theorem of Gotay \cite{Gotay} asserts that actually {\em every} symplectic manifold $(M,\omega)$
into which $C$ embeds as a coisotropic submanifold, such that $\omega_C=\eta$, looks
like the local symplectic model in a neighborhood of $C$:
\end{remark}

\begin{theorem}[Gotay \cite{Gotay}]\label{thm:Gotay}
Let $C$ be a coisotropic submanifold of a symplectic manifold $(M,\omega)$.
There is a symplectomorphism $\psi$ between a tubular neighborhood   of $C$ inside $M$ and a tubular neighborhood of $C$ inside its local symplectic model
$(U,\Omega)$. Moreover, $\psi$ can be chosen such that the restriction of $\psi$ to $C$ is the identity.
\end{theorem}

 Throughout the rest of the paper we fix a local symplectic model $(U,\Omega)$ of the coisotropic submanifold $C$.
Since
 the local symplectic model is a neighborhood of the zero section in a vector bundle $E \to C$,  it
comes equipped with an embedding of the zero section $C$ in $U$, with coisotropic image, as well as with a surjective submersion $\pi \colon U\to C$. 
Recall that 
$E$ is isomorphic to $K^*$, the dual to the kernel of the pre-symplectic structure $\omega_C$.
 To avoid unnecessary confusion about signs, we also assume that $U$ was chosen invariant with respect to fibrewise multiplication by $-1$.
 
 Summarizing, the setting we assume in the rest of the paper is: 
 \begin{center}
\fbox{
\parbox[c]{12.6cm}{\begin{center}
$(M,\omega)$ is a   symplectic manifold,\\
$C$ is a coisotropic submanifolds with induced presymplectic form $\omega_C$,\\
 $(U,\Omega)$ is the local symplectic model,\\
 where $U$ is a neighborhood of the zero section in a vector bundle $E\to C$.
\end{center}
}}
 \end{center}
 
 \section{Deformations of coisotropic submanifolds} \label{section: coisotropic submanifolds}
 
 We set up the problem of deforming a given coisotropic submanifold and {review}
 some relevant results, setting the stage for the subsequent development. In particular, the precise relationship between the deformation problem and the $L_\infty[1]$-algebra
 of Oh and Park \cite{OP,OPPois} is recalled.

 \subsection{The deformation problem}\label{subsection: coisotropic submanifolds - deformation problem}

It is natural to wonder how the ``space of coisotropic submanifolds close to $C$'' looks like, i.e. we ask

\begin{center}
{\em Which deformations of $C$ are coisotropic submanifolds of $(U,\Omega)$?}
\end{center}

\begin{definition}\label{def:coisoDefor}
The space of {\bf coisotropic sections} of $U$ is
$$ \Defor_U(C) := \{ s \in \Gamma(U): \textrm{the graph of }  s \textrm{ is coisotropic inside } (U,\Omega)\}.$$
\end{definition}

We now translate the above question into:

\begin{center}
{\em How can one describe the set $\Defor_U(C)$?}
\end{center}

Theorem \ref{thm:symplectic_case} in Subsection \ref{subsection: coisotropic submanifolds - L_infty} provides an answer to this question.

\subsection{Infinitesimal deformations}\label{subsection: coisotropic submanifolds - infinitesimal deformations}

We discuss the infinitesimal version of the space $\Defor_U(C)$, which
turns out to be closely related to the foliated de Rham complex.

\begin{remark}\label{rem:dP}
Recall from Section \ref{section: pre-symplectic geometry} that the kernel $K$
of the pre-symplectic structure $\omega_C$ on $C$ is involutive, and {that} the associated foliation $\mathfrak{F}$ of $C$
is called the characteristic foliation. One has the following
foliated version of the de Rham complex:
\begin{eqnarray*}
\Omega_\mathfrak{F}(C) &:=&  \Gamma(\wedge K^*),\\
(d_\mathfrak{F}\omega)(s_0,\dots,s_k) &:=& \sum_{i=0}^k (-1)^i s_i(\omega(s_0,\dots, s_{i-1},\widehat{s_{i}},s_{i+1},\dots s_k) \\
&& + \sum_{i<j} (-1)^{i+j}\omega([s_i,s_j],s_1,\dots,\widehat{s_i},\dots, \widehat{s_j},\dots s_k).
\end{eqnarray*}
In Remark \ref{rem:isom},
 we obtained a vector bundle isomorphism
$$ E = TM\vert_C /TC \to K^*,$$
by restricting $\omega^{\sharp}$. This yields an isomorphism $\Gamma(\wedge E) \cong \Gamma(\wedge K^*) = \Omega_\mathfrak{F}(C)$.
The foliated de Rham operator $d_\mathfrak{F}$ then corresponds to the operator
\begin{equation*}
\label{eq:dFP}
 \xi \mapsto P([\Pi,\xi]),
\end{equation*}
where $\xi \in \Gamma(\wedge E)$ is interpreted as a vertical multivector-field that is constant along the fibres of $E$,
and $[\cdot,\cdot]$ is the Schouten-Nijenhuis bracket, see \cite[Proof of Prop 3.5]{OPPois} for more details.
\end{remark}

\begin{remark}\label{rem:mu}
We will show that the formal tangent space to $\Defor_U(C)$ can be identified
with the space of $d_\mathfrak{F}$-closed foliated one-forms on $C$.
To this end, we rewrite the condition for a section $s$ of $U$ to be coisotropic in a more algebraic way.
First, every section $s\in \Gamma(U)$ yields a diffeomorphism
$$ \psi_{-s}: E\to E, \quad (x,e) \mapsto (x,e-s_x),$$
which maps  $\mathrm{graph}(s)$ to the zero section $C\subset E$.

The graph of $s$ is coisotropic with respect to $\Omega$ if and only
if the zero section is coisotropic with respect to $(\psi_{-s})_*\Pi$, where $\Pi$
denotes the Poisson bivector field corresponding to $\Omega$. As discussed in
Section \ref{section: pre-symplectic geometry}, the latter statement can be expressed
by saying that $(\psi_{-s})_*\Pi$ lies in the kernel of the projection map
\begin{equation}\label{eq:Pproj}
P: \mathcal{X}^\bullet(E) \to \Gamma(\wedge E),
\end{equation}
given by restriction to $C$, composed with the projection $\wedge TE\vert_C \to \wedge E$.

Hence, if we define $\mu$ to be the map
$$ \mu: \Gamma(U) \to \Gamma(\wedge^2E), \quad s\mapsto P((\psi_{-s})_*\Pi),$$
a section $s$ will be coisotropic if and only if it is mapped to zero under $\mu$.

The map $\mu$ seems non-local since it involves the symplectic
form away from $C$. However, the symplectic
structure $\Omega$ of the local symplectic model $(U,\Omega)$ is determined by
$\omega_C$.
We will return to this point in Subsection \ref{subsection: coisotropic submanifolds - L_infty}, where we see that 
the equation $\mu(-s)=0$ can in fact
be recovered as the Maurer-Cartan equation of an $L_\infty$-algebra whose structure
maps are multi-differential operators on $C$.
\end{remark}

\begin{proposition}\label{prop:EK*}
 Let $s_t$ be a smooth one-parameter family of sections of $U$
which starts at the zero section $s_0=0$.
 Then
 $$ \frac{\partial}{\partial t}\vert_{t=0} \mu(s_t) = - d_\mathfrak{F}(\frac{\partial}{\partial t}\vert_{t=0} s_t)$$
under the identification $E\cong K^*$.
\end{proposition}

\begin{proof}
Consider the one-parameter family of diffeomorphisms $\psi_{-s_t}\colon E\to E$.
The corresponding time-dependent vector field is $Y_t:=-\frac{\partial}{\partial t} s_t$, a vertical vector field which is constant on each fibre of $E$. Using this and
the definition of $\mu$, we see that $\frac{\partial}{\partial t}\vert_{t=0}\mu(s_t)$
equals the image under the projection
 $P: \chi^\bullet(E) \to \Gamma(\wedge E)$ of
$\mathcal{L}_{\frac{\partial}{\partial t}\vert_{t=0}s_t}\Pi=-[\Pi,\frac{\partial}{\partial t}\vert_{t=0}s_t]$.
By Remark \ref{rem:dP} this is exactly the formula for the image of $\frac{\partial}{\partial t}\vert_{t=0}s_t$ under $\-d_\mathfrak{F}$,
if we apply the identification $E\cong K^*$.
\end{proof}

\begin{corollary}
Let $s_t$ be a smooth one-parameter family of coisotropic sections of $E$ with
$s_0=0$.
Then $\frac{\partial}{\partial t}\vert_{t=0}s_t$ is closed with respect to $d_\mathfrak{F}$.
\end{corollary}
\begin{proof}
We have $\mu(s_t)=0$ for all $t$ by Remark \ref{rem:mu}, hence the statement follows from Proposition \ref{prop:EK*}.
\end{proof}

\begin{remark}\label{rem:Omega1cl}
Proposition \ref{prop:EK*} identifies the space of closed elements of $\Omega_\mathfrak{F}^1(C)$
with the formal tangent space to $\Defor_U(C)$ at $C$, where the formal tangent space is defined as the space
of solutions to the linearized equation. 
We point out that it is known  that not all cohomology classes
of $H^1_{\mathfrak{F}}(C)$ can be realized through one-parameter families of deformations,
see \cite{coisoemb,OP}. 
\end{remark}

\subsection{Oh and Park's $L_\infty[1]$-algebra} \label{subsection: coisotropic submanifolds - L_infty}

We recall the $L_\infty[1]$-algebra associated to $C$ \cite{OP,CaFeCo2}.\footnote{The reader is referred to  \cite[Appendix D]{LOTVcoiso} for a proof that the
construction from \cite{OP} coincides with the one from \cite{CaFeCo2}, specialized to the symplectic case.}

\begin{definition}  An {\bf $L_\infty[1]$-algebra} is a $\ZZ$-graded vector space $W$, 
equipped with a collection of graded symmetric brackets $(\lambda_k\colon W^{\otimes k} \longrightarrow W)_{k\ge1}$ of degree $1$ which satisfy a collection of quadratic relations \cite{LadaStasheff}, called higher Jacobi identities.

The {\bf Maurer-Cartan series} of a degree zero element $\beta\in W$ is the infinite sum
$$ \mathsf{MC}(\beta) := \sum_{k\ge 1} \frac{1}{k!}\lambda_k(\beta^{\otimes k}).$$ We say that $\beta$ is a  Maurer-Cartan element if its Maurer-Cartan series converges to zero\footnote{...with respect to a suitable topology. For the specific examples of $L_\infty[1]$-algebras with which we will be concerned later on, we will make this precise.}.
We denote the set of all Maurer-Cartan elements of $W$ by $\mathsf{MC}(W)$.
\end{definition}

\begin{remark}
 In order to describe the $L_\infty[1]$-algebra associated to the coisotropic submanifold $C$ of $(U,\Omega)$  as explicitly as possible, we consider the Poisson structure $\Pi$ associated to $\Omega$.
 As explained in Section \ref{section: pre-symplectic geometry}, the  coisotropicity of $C$ is equivalent to $P(\Pi)=0$, where
  $$P: \chi^{\bullet}(E)\to \Gamma(\wedge E)$$ is as in Equation \eqref{eq:Pproj}.

  As shown in \cite{OP} and \cite{CaFeCo2}, the space $\Gamma(\wedge E)[1]$ is equipped with a canonical $L_\infty[1]$-algebra structure.
We denote the structure maps of this $L_\infty[1]$-algebra by
$$ \lambda_k\colon \Gamma(\wedge E)[1]^{\otimes k} \to \Gamma(\wedge E)[1].$$
The evaluation of $\lambda_k$ on $s\otimes\cdots\otimes s$ for $s \in \Gamma(E)$ yields
 \begin{equation}\label{eq:multibrackets}
 \lambda_k(s,\dots,s) := P\big([[\dots[\Pi,s],s]\dots],s] \big),
\end{equation}
where $s$ is interpreted as a fibrewise constant vertical vector-field on $E$. Hence the Maurer-Cartan series of $s$ reads
$\mathsf{MC}(s)=P(e^{[\cdot,s]}\Pi)$.
 
The following result, which is -- partly in an implicit manner -- contained in \cite{OP}, is essentially \cite[Thm. 2.8]{OPPois}. It relies on the fact that the Poisson bivector field associated to $\Omega$ is analytic in the fibre direction, which is true
thanks to \cite[Cor. 2.7]{OPPois}. {In
 \cite{OPPois}, such bivector fields are called fibrewise entire and most of the subsequent discussion carries over to such Poisson bivector fields.
We refer the interested reader to
Appendix \ref{appendix: fibrewise entire} for more details.}
\end{remark}

\begin{theorem}\label{thm:symplectic_case}
Consider the $L_\infty[1]$-algebra  
$\Gamma(\wedge E)[1]$ associated to the coisotropic submanifold $C$.
For any $s \in \Gamma(E)$ such that $\graph(s)$ is contained in $(U,\Omega)$, 
the Maurer-Cartan series $\mathsf{MC}(-s)$ is pointwise convergent.
Furthermore, for any such $s$ the following two statements are equivalent:
\begin{enumerate}
 \item $\graph(s)$ is a coisotropic submanifold of $(U,\Omega)$.
 \item The Maurer-Cartan series $\mathsf{MC}(-s)$ converges to zero (in the sense of pointwise convergence).
\end{enumerate}
\end{theorem}

\begin{remark}\label{rem:sminuss}
In other words, if we restrict attention to those sections whose graphs lie inside $U$,  { the map $s\mapsto -s$} restricts to a bijection between
the set of coisotropic sections
$$ \Defor_{U}(C) := \{ s \in \Gamma(U): \textrm{the graph of } s \textrm{ is coisotropic inside } (U,\Omega)\}$$
from Subsection \ref{subsection: coisotropic submanifolds - deformation problem},
and  
$$\mathsf{MC}_U(\Gamma(\wedge E)[1]):=\{\textrm{Maurer-Cartan elements  
of   $\Gamma(\wedge E)[1]$ whose graphs lie in $U$}\}.$$
  
Notice that the first structure map $\lambda_1$ of the $L_\infty[1]$-algebra  
$\Gamma(\wedge E)[1]$ coincides with the foliated de Rham differential
$d_\mathfrak{F}$ under the isomorphism $\Gamma(\wedge E) \cong \Omega_\mathfrak{F}(C)$.
We could -- a posteriori -- use this fact to recover the infinitesimal description of $\Defor_U(C)$ which we obtained in
Subsection \ref{subsection: coisotropic submanifolds - infinitesimal deformations}.
\end{remark}

\section{Hamiltonian diffeomorphisms}\label{section: Ham diffeos}

In this section we investigate the action of Hamiltonian diffeomorphisms on the space of coisotropic submanifolds. 
More precisely, we provide
a description of the induced equivalence relation on the space of coisotropic sections.
As the main result, we show that for compact coisotropic submanifolds this equivalence relation coincides 
with the gauge-equivalence in Oh and Park's
$L_\infty[1]$-algebra. This result was obtained independently by
L{\^e}, Oh, Tortorella and  Vitagliano
 in \cite[Corollary 4.24]{LOTVcoiso}.

\subsection{The deformation problem}\label{subsection: Ham diffeos - deformation problem}

Recall that by Definition \ref{def:coisoDefor} a section $s$ of {$\pi: U \to C$} is called coisotropic if $\graph(s)$  is a coisotropic submanifold
of $(U,\Omega)$, and that we denote the set of all such sections by $\Defor_U(C)$.

\begin{definition}\label{def:hamsymeq}
Two coisotropic sections $s_0$ and $s_1$ are called {\bf Hamiltonian equivalent} if there is 
 a family of coisotropic sections $s_t$, agreeing with the given
ones at $t=0$ and $t=1$, and an isotopy of Hamiltonian diffeomorphisms $\phi_t$
 such that $\phi_t$ maps the graph of $s_0$ to the graph of $s_t$ for all $t\in [0,1]$.
\end{definition}

\begin{remark}
To be more precise, we assume that we are given a locally defined Hamiltonian isotopy, i.e.
a family of diffeomorphisms between open subsets of $U$, generated by a family
of locally defined Hamiltonian vector fields, 
which maps $\graph(s_0)$
onto $\graph(s_t)$. 
\end{remark}

It is straight-forward to check that Hamiltonian equivalence actually defines an
equivalence relations on the set $\Defor_E(C)$, which we denote by $\sim_\Ham$.
We refer the interested reader to \cite[Lemma 1]{FlorianModuli} for a proof of this fact.
It is natural to wonder about the equivalence classes {of} $\sim_\Ham$,
so we define:

\begin{definition}\label{definition: ham equivalence}
The {\bf Hamiltonian moduli space of coisotropic sections} is the set
 $$\Moduli^\Ham_U(C):= \Defor_U(C) / \sim_\Ham.$$
\end{definition}
We ask:

\begin{center}
 {\em How can one describe the set $\Moduli^\Ham_U(C)$?}
\end{center}

Theorem \ref{theorem: Hamiltonian equivalence} of Subsection \ref{subsection: Ham diffeos - equivalences match}
provides an answer in terms of the $L_\infty[1]$-algebra of Oh and Park.

\subsection{Infinitesimal moduli}\label{subsection: Ham diffeos - infinitesimal moduli}

We discuss the infinitesimal version of $\Moduli^\Ham_U(C)$.
In particular, we argue that the formal tangent space to $\Moduli^{\Ham}_U(C)$
at the equivalence class of the zero-section $C$
is given by the first foliated cohomology $H^1_\mathfrak{F}(C)$, with $\mathfrak{F}$
the characteristic foliation of the pre-symplectic structure on $C$.
The results of this subsection can be recovered -- via specialization to the symplectic case -- from the results obtained by L\^e and Oh, \cite[Subsection 6.3]{OhLe}, who studied deformations
of coisotropic submanifolds in locally conformal symplectic manifolds.

\begin{remark}
Let $(s_t)_{t\in [0,1]}$ be a family of coisotropic sections that starts at the zero-section.
In Subsection \ref{subsection: coisotropic submanifolds - infinitesimal deformations} we saw
that $\frac{\partial s_t}{\partial t}\vert_{t=0} \in \Gamma(E)$ lies in the kernel of the complex
$ (\Gamma(\wedge E), P([\Pi,-]))$ and that the latter is isomorphic to the foliated 
de Rham complex $(\Omega_\mathfrak{F}(C),d_\mathfrak{F})$.
\end{remark}

\begin{proposition}\label{prop:trivialclass}
Suppose that $(s_t)_{t\in [0,1]}$ is a family of coisotropic sections that starts at the zero-section
and is trivial under Hamiltonian equivalence, i.e. there is an Hamiltonian isotopy $\phi_t$ such that
the graph of $s_t$ coincides with the image of the zero section under $\phi_t$.

Then the cohomology class of $\frac{\partial s_t}{\partial t}\vert_{t=0}$ in $H^1_\mathfrak{F}(C)$ is trivial.
\end{proposition}

\begin{proof}
Suppose that $\phi_t$ is generated by the family of Hamiltonian vector fields $X_{H_t}$.
We can write $\frac{\partial s_t}{\partial t}\vert_{t=0}$
as $ P(X_{H_0}) = P([\Pi,H_0])$ (see Lemma \ref{lem:techeasy} later on).
We observe that the latter expression equals $P([\Pi,H_0\vert_C])$,
because $\Pi^{\sharp}\vert_C$ maps the co-normal bundle to the tangent bundle $TC$, whose
sections lie in the kernel of $P$.
As a consequence, the cohomology class of $\frac{\partial s_t}{\partial t}\vert_{t=0}$
equals the cohomology class of $P([\Pi,H_0\vert_C])$, which is trivial.
Now apply the isomorphism between $\Gamma(\wedge E)$ and the foliated de Rham complex from Remark \ref{rem:mu}.
\end{proof}

\begin{remark}\label{TMHam}

For every $f\in \cC^{\infty}(C)$, let $\phi_t$ be the flow of the Hamiltonian vector field $X_{\pi^*f}$, and   $(s_t)_{t\in [0,\epsilon)}$ the family of coisotropic sections determined by  $\mathrm{graph}(s_t)=\phi_t(C)$.
Then the proof of Proposition \ref{prop:trivialclass} shows that
$\frac{\partial s_t}{\partial t}\vert_{t=0}$ corresponds to $d_{\mathfrak{F}}f$ under the isomorphism $\Gamma(E)\cong \Omega^1_{\mathfrak{F}}(C)$.
Hence we can refine Proposition \ref{prop:trivialclass} as follows: the formal tangent space of the set of coisotropic sections which are   trivial under Hamiltonian equivalence is precisely $\Omega^1_{\mathfrak{F}, \mathrm{exact}}(C)$.

This and 
Remark \ref{rem:Omega1cl} imply that the formal tangent space at  zero  to
$\Moduli^\Ham_U(C)$ is   $H^1_\mathfrak{F}(C)$. 
In the special case of $C$ Lagrangian, this reduces to the first de Rham cohomology $H^1(C)$ of $C$, as expected.
\end{remark}

\subsection{Gauge-equivalence}\label{subsection: Ham diffeos - equivalences}

\begin{remark}
Convergence issues aside, every $L_\infty[1]$-algebra $W$ comes along with a (singular) foliation on its   set of Maurer-Cartan elements $\mathsf{MC}(W)$.
On  $W_0$, the elements  of degree $0$, there is a distribution  generated by vector fields $V_\gamma$  associated to elements $\gamma$ of degree $-1$. At the point $\beta \in W_0$, the vector field $V_\gamma$ reads
$$ \lambda_1(\gamma) + \lambda_2(\gamma,\beta) + \frac{1}{2!}\lambda_3(\gamma,\beta,\beta) + \frac{1}{3!}\lambda_4(\gamma,\beta,\beta,\beta) +\cdots.$$
The vector fields $V_\gamma$ are tangent to
$\mathsf{MC}(W)$  and {they form an involutive distribution there}, hence we obtain a canonical equivalence relations on $\mathsf{MC}(W)$:

\end{remark}

 \begin{definition}\label{def:eqMC}
 Two Maurer-Cartan elements $\beta_0$ and $\beta_1$ of an $L_\infty[1]$-algebra $W$ are {\bf gauge-equivalent} if there is a one-parameter family $\gamma_{t}$ of degree $-1$ elements of $W$ and a  one-parameter family $\beta_{t}$ of degree zero elements of $W$, agreeing with the given ones at $t=0$ 
  and $t=1$, such that
\begin{equation*}
\frac{\partial}{\partial t} \beta_{t}=\lambda_1(\gamma_{t})+ \lambda_2(\gamma_{t},\beta_{t} )+\frac{1}{2!}\lambda_3(\gamma_t,\beta_{t},\beta_{t})+\frac{1}{3!}\lambda_4(\gamma_t,\beta_{t},\beta_{t},\beta_{t})+\dots
\end{equation*}
We presuppose that $W$ is equipped with a suitable topology and that the right-hand side of the above
equation converges.
\end{definition}

We apply this to  the $L_\infty[1]$-algebra
structure on $\Gamma(\wedge E)[1]$ from Subsection \ref{subsection: coisotropic submanifolds - L_infty}. We are interested in $\mathsf{MC}_U(\Gamma(\wedge E)[1])$, the
Maurer-Cartan elements  of   $\Gamma(\wedge E)[1]$ whose graphs lie in $U$
(see Remark \ref{rem:sminuss}). We define an equivalence relation on $\mathsf{MC}_U(\Gamma(\wedge E)[1])$ as in Def.
\ref{def:eqMC}, but additionally requiring that the one-parameter family of degree zero elements $\beta_{t}$   
consists of sections of $U$ (rather than $E$).
We use the bijection $\Defor_{U}(C)\cong \mathsf{MC}_U(\Gamma(\wedge E)[1]), s\mapsto -s$ 
described in Remark \ref{rem:sminuss} to transport the above equivalence relation to $\Defor_{U}(C)$:

\begin{definition}\label{def:gaugeGammaU}
Two coisotropic sections $s_0$ and $s_1$ are called {\bf gauge-equivalent},  $s_0 \sim_{\gauge}  s_1$,  if $-s_0$ and $-s_1$ are equivalent  elements (in the sense above) of $\mathsf{MC}_U(\Gamma(\wedge E)[1])$.\end{definition}

\begin{remark}
We make the equivalence relation $\sim_{\gauge}$ more explicit.
Two elements $s_0$ and $s_1$ in   $\Defor_U(C)$ are declared gauge-equivalent
if there is a smooth one-parameter family $s_t$ in $\Gamma(U)$,
coinciding with $s_0$ and $s_1$ at the endpoints, 
 such that
\begin{eqnarray*}\label{eq:ddtP}
\frac{\partial}{\partial t}(-s_t) &=& P([\Pi,\pi^*f_t]) + P([[\Pi,\pi^*f_t],-s_t]) + \frac{1}{2!}P([[[\Pi,\pi^*f_t],-s_t],-s_t]) + \cdots\\
&=& P(e^{[\cdot,-s_t]}X_{\pi^*f_t}).\nonumber
\end{eqnarray*}
Here $-s_t$ is interpreted as a family of fibrewise constant vertical vector field and
$f_t$ is a one-parameter family of smooth functions on $C$.
Observe that the latter can be seen as a one-parameter family of degree $-1$ elements of
the $L_\infty[1]$-algebra $\Gamma(\wedge E)[1]$.
To rewrite the condition in more geometric terms, recall that for $s\in \Gamma(E)$, $\psi_{s}$ is the diffeomorphism
of $E$ that consists of fibrewise addition with $s$. Moreover, let $\pr_s$
be the projection of $TE \vert_{\graph(s)}$ onto the vertical part of $TE$ along $T\graph(s)$.

We now compute
 \begin{align*}\label{eq:PeprV}
P( e^{[\cdot,-s_t]}X_{\pi^*f_t})&=P\big((\psi_{-s_t})_*X_{\pi^*f_{t}}\big)
=\pr_{0}\big((\psi_{-s_t})_*(X_{\pi^*f_t}\vert_{\graph(s_t)})\big)\\
&=(\psi_{-s_t})_*\big(\pr_{s_t}(X_{\pi^*f_t}\vert_{\graph(s_t)})\big)
=\pr_{s_t}(X_{\pi^*f_t}\vert_{\graph(s_t)}).
\end{align*}
We use \cite[Prop. 1.15]{OPPois} in the first equality\footnote{\cite[Prop. 1.15]{OPPois}
is stated for bivector fields, but it carries over immediately to the case of vector fields.}, which applies since the vector field $X_{\pi^*f_{t}}$ is fibrewise entire in the terminology of \cite{OPPois}. In the last equality we used the fact that $\psi_{-s_t}$ maps $\graph(s_t)$ to the zero section $C$  and preserves the fibres of the projection $\pi: U \to C$.
\end{remark}

After reversing the signs in front of  $f_t$, this shows:

\begin{proposition}\label{algequiv}
Elements $s_0$ and $s_1$ of $\Defor_U(C)$ are gauge-equivalent if and only if
 there is
  a one-parameter family $s_t\in \Gamma(U)$, agreeing with $s_0$ and $s_1$ at the endpoints, 
  and a one-parameter family $f_t\in \mathcal{C}^{\infty}(C)$ such that
\begin{equation}\label{eq:prv}
\frac{\partial}{\partial t}s_t=\pr_{s_t}(X_{\pi^*f_t}\vert_{\graph(s_t)})
\end{equation}
holds for all $t\in [0,1]$.
\end{proposition}

\subsection{Technical Lemmata}

We establish some technical lemmata that we use subsequently to relate various
notions of equivalence between coisotropic sections.

\begin{remark}
Throughout this subsection, $A$ denotes a vector bundle over a smooth manifold $M$.
Given  a section $s$ of $A$ and a point $y\in \graph(s)$, we have a splitting $T_yA=V_y\oplus T_y \graph(s)$ of the tangent space to $A$ at $y$, where $V:=\ker(d\pi)$ is the vertical bundle. We will denote by $\pr_{s}$ the projection $T_yA\to V_y$ with kernel $T_y \graph(s)$.
\end{remark}

 \begin{lemma}\label{lem:techeasy}
Let $X_t$ be a one-parameter family of vector fields on $A$, and $\phi_{t}$ its flow. Moreover, let $s_t$ be a one-parameter family of sections of $A$ such that
$$  \graph(s_t)= \phi_{t}(\graph(s_{0}))$$
holds for all $t\in [0,1]$.

Then $s_t$ satisfies the equation
 \begin{equation*}
\frac{\partial}{\partial t} s_t= \pr_{s_t}X_{t},\;\;\;\;\;\forall t\in [0,1],
\end{equation*}
which we see as an equality of sections of $V\vert_{\graph(s_t)}$.
\end{lemma}

\begin{proof}
If we define $\psi_t$ to be the isotopy of $M$ given by
$\pi \circ \phi_t \circ s_0$, we have
$$ s_t =\phi_t \circ s_{0} \circ (\psi_t)^{-1}: M\to A.$$ Evaluating at $x\in M$ and taking the time derivative we obtain
$$\frac{\partial}{\partial t} (s_t(x))= X_t \vert_{s_t(x)}
+ (\phi_t)_*(s_{0})_* \frac{\partial}{\partial t} ((\psi_t)^{-1}(x)).$$
We finish noticing that the last summand is tangent to  $\phi_t(\graph(s_{0}))=\graph(s_t)$, and that $\frac{\partial}{\partial t} (s_t(x))$ lies in $V_{s_t(x)}$.
\end{proof}

The following Lemma, whose (geometric) proof  was communicated to us by Luca Vitagliano, is a converse to Lemma \ref{lem:techeasy}.

 \begin{lemma}\label{lem:techhard}
Let $X_t$ be a one-parameter family of vector fields on $A$, and $\phi_t$ its flow, assumed to exist for all $t\in [0,1]$. Suppose $s_t$ is a one-parameter family of sections of $A$ that satisfies
\begin{equation}\label{diffeq}
\frac{\partial}{\partial t} s_t= \pr_{s_t}X_t,\;\;\;\;\;\forall t\in [0,1].
\end{equation}
Then the family of submanifolds $ \graph(s_t)$ coincides with $\phi_t(\graph(s_{0}))$ for all $t\in [0,1]$.
\end{lemma}

\begin{proof}
We work on the vector bundle $A\times [0,1]\to M\times [0,1]$, and denote by $t$ the standard coordinate on the $[0,1]$-factor. Define $\widehat{s}\in \Gamma(A\times [0,1])$ by $$\widehat{s}(x,t)=(s_t(x),t)$$ and the vector field $\widehat{X}$ on $A\times [0,1]$ by $$\widehat{X}\vert_{(y,t)}=(X_{t})\vert_{y}+\pd{t}.$$ Notice that the flow $\varphi_t$ of $\widehat{X}$ takes $(y,0)$ to $(\phi_t(y),t)$ for all $y\in A$.

The key observation is that the vector field $\widehat{X}$ is tangent to the submanifold $\graph(\widehat{s})$. To this end we compute
\begin{equation*} \frac{d}{dt} \widehat{s}(x,t)
=\frac{d}{dt} s_{t}(x)+\pd{t}
= (X_{t})\vert_{s_{t}(x)}-v+\pd{t}=
\widehat{X}|_{({s_{t}(x)},{t})} -v
\end{equation*}
for some vector $v\in T_{s_{t}(x)}(\graph(s_{t}))$,  making use of equation \eqref{diffeq} in the second equality. This implies that  $\widehat{X}|_{({s_{t}(x)},{t})}=\frac{d}{dt} \widehat{s}(x,t)+v$ is the sum of two vectors tangent to $\graph(\widehat{s})$.

Hence the flow $\varphi_t$ of $\widehat{X}$ maps $\graph(\widehat{s}\vert_{M\times \{0\}})=\graph(s_{0})\times \{0\}$ to 
 $\graph(\widehat{s}\vert_{M\times \{t\}})=\graph(s_{t})\times \{t\}$. On the other hand, we saw above that  $\varphi_t$ maps $\graph(s_{0})\times \{0\}$ to $\phi_t(\graph(s_{0}))\times\{t\}$.
\end{proof}

In Lemma \ref{lem:techhard} we assume that the flow of $X_t$ is defined on the interval $[0,1]$. We now show that this assumption can be replaced by asking that the base $M$ of the vector bundle be compact.

\begin{lemma}\label{lem:techhard2}
 Let $\pi\colon A\to M$ be a vector bundle over a compact base $M$. Let $X_t$ be a one-parameter family of vector fields on $A$ and $s_t$ a one-parameter family of sections of $A$ that satisfies
\begin{equation*}
\frac{\partial}{\partial t} s_{t}= \pr_{s_t}X_{t},\;\;\;\;\;\forall t\in [0,1].
\end{equation*}
Then the flow lines of $X_t$ starting at $\graph(s_0)$ exist for $t\in [0,1]$
and the equality 
 \begin{equation*}
 \graph(s_t)= \phi_{t}(\graph(s_{0}))
\end{equation*}
 holds.
\end{lemma}

\begin{proof}
Fix an auxiliary fibre metric on $A$. We let $K\subset A$ be the compact subset given by
all vectors of length less than or equal to $l+\delta$ for some $\delta>0$, where
$$ l := \max_{x\in M,\ t \in[0,1]} (\vert\vert s_t(x)\vert\vert).$$

 Let $\varphi$ be a function on $A$ with compact support, and so that $\varphi\vert_K\equiv 1$. Then $(\varphi X_t)_{t\in [0,1]}$ is a time-dependent vector field whose integral curves are defined for all times.  Let $T$ be the maximal element of $[0,1]$ such that $\graph(s_t)=\phi_t(\graph(s_0))$ holds for all $t\in [0,T]$. 
 Suppose $T<1$.
There is $\epsilon>0$ such that $\phi_t(\graph(s_{0}))\subset K$ for all $t\in [0,T+\epsilon]$. But since the one-parameter families $X_t$ and $\varphi X_t$ agree on $K$, we see as in Lemma \ref{lem:techhard} that 
 $\graph(s_t) = \phi_t(\graph(s_0))$ actually holds
  for all $t\in [0,\mathrm{min}\{1,T+\epsilon\}]$, which is a contradiction.
\end{proof}

\begin{remark}
The compactness assumption in Lemma \ref{lem:techhard2} can not be omitted, as the following counter-example shows. Take a non-compact manifold
$M$, a vector field $X$ on $M$ whose flow is not defined on the whole of $[0,1]$. Take the trivial bundle $A:=M\times [0,1]$
and let $X_t$ be the horizontal lift of $X$ to $A$. Moreover, let $\graph(s_t)$ be $M\times \{0\}$. Notice that 
$\frac{\partial}{\partial t} s_t$ and
 $\pr_{s_t}X_{t}$ agree, since they both vanish identically.
\end{remark}

\subsection{Hamiltonian equivalence $=$ gauge-equivalence} \label{subsection: Ham diffeos - equivalences match}

Our aim is to compare the two equivalence relations $\sim_{\gauge}$ and $\sim_{\Ham}$
on $\Defor_U(E)$. As an intermediate notion we introduce:

\begin{definition}\label{def:baseham}
 One can restrict Hamiltonian equivalence $\sim_{\Ham}$ by only allowing Hamiltonian flows
 generated by functions of the type $\pi^*f$, with $f\in \mathcal{C}^\infty(C)$.
 We call the resulting equivalence relation {\bf base Hamiltonian equivalence} and denote it by
 $\sim_{b\Ham}$.
\end{definition}

\begin{proposition}\label{prop:implications1}
The following chain of implications holds between the three equivalence 
relations on $\Defor_U(C)$:
$$
\xymatrix{
\textrm{base Hamiltonian equivalence } \sim_{b\Ham} \ar@{=>}[d]^{(1)}  \\
 \textrm{Hamiltonian equivalence } \sim_{\Ham}\ar@{=>}[d]^{(2)} \\
\textrm{gauge-equivalence } \sim_{\gauge}.}
$$
\end{proposition}

\begin{proof}
Implication (1) is clear, so we pass on to implication (2).
Let $s_t$ be a smooth family of coisotropic sections of $U$
and suppose that $H_t$ is a smooth family of functions on $U$ such that the Hamiltonian
flow $\phi^{H_t}_t$ of $H_t$ maps $\graph(s_0)$ to $\graph(s_t)$.  
By Lemma \ref{lem:techeasy}, this implies that the equation
\begin{equation*}
\frac{\partial}{\partial t} s_t= \pr_{s_t}X_{H_t},
\end{equation*}
holds for all $t\in [0,1]$.
Define $f_t\in \mathcal{C}^{\infty}(C)$ to be $H_t\circ s_t$.

Observe that $\pr_{s_t}(X_{H_t} -X_{\pi^*f_t})$ is zero
since $H_t - \pi^*f_t$ vanishes on $\graph(s_t)$
and consequently $X_{H_t -\pi^*f_t} = X_{H_t} - X_{\pi^*f_f}$ gets mapped to $T\graph(s_t)$ under $\Pi^{\sharp}$, since $\graph(s_t)$ is coisotropic.
We conclude that the equation
\begin{equation*}
\frac{\partial}{\partial t} s_t= \pr_{s_t}X_{H_t} = \pr_{s_t}(X_{\pi^*f_t})
\end{equation*}
holds. By Proposition \ref{algequiv} we have that $s_0$ and $s_1$ are
gauge-equivalent as claimed.
\end{proof}

Under the assumption that  $C$ is compact, we can ``close the circle'' of the implications of Proposition \ref{prop:implications1}:

\begin{proposition}\label{prop:compacteq}
Suppose $C$ is compact coisotropic submanifold. Then the following implication holds for the local symplectic model
of $C$:
$$
\xymatrix{
\textrm{gauge-equivalence }\sim_{\gauge}\ar@{=>}[d] \\
\textrm{base Hamiltonian equivalence }\sim_{b\Ham}.}$$
\end{proposition}

\begin{proof}
Suppose that $s_0$ and $s_1$ of $\Defor_U(C)$ are gauge-equivalent.
This means that there is a one-parameter family $s_t$ in $\Defor_U(C)$
and a one-parameter family of functions $f_t$ on $C$ such that
$$
\frac{\partial}{\partial t} s_t = \pr_{s_t} X_{\pi^*f_t}
$$
holds for all $t\in [0,1]$.

The compactness of $C$  allows us to apply Lemma \ref{lem:techhard2}, which states that the flow $\phi_t$ of $X_{\pi^*f_t}$ exists for all $t\in [0,1]$ and    
  indeed maps $\graph(s_0)$ to $\graph(s_t)$. 
\end{proof}

\begin{remark}
When $C$ is a Lagrangian submanifold, {Hamiltonian equivalence } implies base Hamiltonian equivalence without any compactness assumption: this follows from 
Proposition \ref{prop:implications1} and Proposition \ref{prop:compacteq}, noticing that in the latter in the Lagrangian case no compactness is necessary, for $X_{\pi^*f_t}$ is a vertical vector field on $U\subset T^*C$.
In particular, if $(\phi_t)_{t\in[0,1]}$ is an isotopy by Hamiltonian diffeomorphisms mapping the zero section $C$ to sections of $U$ for all $t\in [0,1]$, then $\phi_1(C)$ is the graph of an exact 1-form on $C$. This 
 is in agreement with \cite[Proposition 9.33]{McSalTop}.
\end{remark}

Combining Proposition \ref{prop:implications1} and Proposition \ref{prop:compacteq} we arrive
at the main result of this section:

\begin{theorem}\label{cor:3agree} 
Let $C$ be a compact coisotropic submanifold with local symplectic model $(U,\Omega)$.
The equivalence relations on
$$ \Defor_U(C) := \{s\in \Gamma(U): s \textrm{ is coisotropic} \}$$
given by
\begin{itemize}
\item Hamiltonian equivalence $\sim_{\Ham}$ (Definition \ref{def:hamsymeq}) and
\item gauge-equivalence $\sim_{\gauge}$
(Definition \ref{def:gaugeGammaU}, see also Proposition \ref{algequiv})
\end{itemize}
coincide.
\end{theorem}

As a consequence we obtain the following result:

\begin{theorem} \label{theorem: Hamiltonian equivalence}
Let $C$ be a compact coisotropic submanifold with local symplectic model $(U,\Omega)$.
The bijection
$$\Defor_U(C) \cong \mathsf{MC}_U(\Gamma(\wedge E)[1])$$
descends to a bijection
$$ \Moduli_U^{\Ham}(C) := \Defor_U(C)/\sim_\Ham \cong \mathsf{MC}_U(\Gamma(\wedge E)[1]) /\sim_{\gauge}.$$
\end{theorem}

\begin{remark}
\hspace{0cm}
\begin{enumerate}
 \item One could use Theorem \ref{theorem: Hamiltonian equivalence} to rederive
 the infinitesimal description of $\Moduli_U^{\Ham}(C)$ from Subsection \ref{subsection: Ham diffeos - infinitesimal moduli}
 by linearizing the Maurer-Cartan equation and the gauge-equivalence.
 \item A description of $\Moduli_U^{\Ham}(C)$ similar to Theorem \ref{theorem: Hamiltonian equivalence} was obtained
in \cite{FlorianModuli}. There the differential graded Lie algebra
associated to the BFV-complex was used to encode deformations of $C$
and the action of Hamiltonian diffeomorphisms. The BFV-complex has the advantage
that it {works for arbitrary Poisson structures}, unlike  the $L_\infty[1]$-algebra from \cite{OP} and \cite{CaFeCo2}.
The drawbacks of the approach relying on the BFV-complex is that one needs
to single out the geometrically relevant Maurer-Cartan elements by hand and is forced to deal with
symmetries of symmetries.
\end{enumerate}
\end{remark}

\section{Symplectomorphisms}\label{section: symplectomorphisms}

Next we consider the action of symplectomorphisms on the space of coisotropic sections,
which we encode by an equivalence relation $\sim_\Sym$ on the space of coisotropic sections
$\Defor_U(C)$.
In the search for an interpretation of $\sim_\Sym$ in terms of Oh and Park's $L_\infty[1]$-algebra,
we are led to reconsider Voronov's derived bracket construction \cite{vor1,vor2}.

\subsection{The deformation problem}\label{section: symplectomorphisms - deformation problem}
 
Let $C$ be a coisotropic submanifold with local symplectic model $(U,\Omega)$.

\begin{definition}\label{def:symeq}
Two coisotropic sections $s_0$ and $s_1$ of $U$ are called {\bf symplectic equivalent},
$s_0 \sim_\Sym s_1$ if there is a family of coisotropic sections $s_t \in \Gamma(U)$, agreeing with the given
ones at $t=0$ and $t=1$,
and an isotopy of local symplectomorphisms $\phi_t$ such that $\phi_t$ maps $\graph(s_0)$
to $\graph(s_t)$ for all $t\in [0,1]$.  
\end{definition}

\begin{remark}
As for Hamiltonian equivalence, it is straight-forward to check that $\sim_\Sym$ is in fact
an equivalence relation. We define the {\bf symplectic moduli space of coisotropic sections} to be the set
 $$ \Moduli^\Sym_U(C) := \Defor_U(C) / \sim_\Sym.$$
\end{remark}

Our aim is to answer
\begin{center}
{\em How can one describe the set $\Moduli^\Sym_U(C)$?}
\end{center}
which we will achieve in Theorem \ref{theorem: symplectic equivalence} of Subsection \ref{subsection: symplectomorphisms - equivalences coincide}.

\subsection{Infinitesimal moduli}\label{subsection: symplectomorphisms - infinitesimal moduli}

We first consider the infinitesimal counterpart of $\Moduli^\Sym_U(C)$.
We argue -- see Remark \ref{rem:cokernel} -- that the formal tangent space to $\Moduli^\Sym_U(C)$
at the equivalence class of the zero-section $C$
is given by the cokernel of a certain map $r: H^1(C) \to H^1_\mathfrak{F}(C)$.

\begin{remark}
\hspace{0cm}
\begin{enumerate}
 \item Recall that every coisotropic submanifold $C$ comes equipped with a pre-symplectic structure
$\omega_C$, whose kernel $K$ is an involutive distribution. The
 corresponding foliation of $C$ is denoted by $\mathfrak{F}$.
Restriction to $K$ yields a chain map
$$ r: \Omega(C) \to \Omega_\mathfrak{F}(C)$$
between the ordinary and the foliated de Rham complex of $C$.
\item As we observed in Subsection \ref{subsection: coisotropic submanifolds - infinitesimal deformations},
$\Omega_\mathfrak{F}(C)$ is isomorphic to $\Gamma(\wedge E)$, equipped with the differential
$ P([\Pi,\cdot])$, where $P$ is the projection from multivector-fields on $E$ onto $\Gamma(\wedge E)$.
\end{enumerate}
\end{remark}

\begin{lemma}
Let $C$ be a coisotropic submanifold of $(E,\omega)$  {with inclusion map $\iota$}.
Given $\beta \in \Omega^1(E)$, denote by $X_{\beta}$ the unique vector field on $E$ which satisfies
$$ i_{X_{\beta}} \omega = \beta.$$
Then the triangle
$$
\xymatrix{
& & \beta \ar@{|->}[ddll]\in & \Omega^1(E) \ar[ddll] \ar[ddrrr]^{r\circ \iota^*}& && \\
& && & &&\\
P(X_{\beta}) \in & \Gamma(E)\ar[rrrrr]_{\cong} & && && \Omega^1_\mathfrak{F}(C).
}
$$
commutes.
\end{lemma}

\begin{proof}

The identification $E\cong K^*$ from Section \ref{section: pre-symplectic geometry}, which is used in the bottom map of the above diagram, 
maps  $e\in E_x$ to $\omega^{\sharp}(e)|_{K_x}$.
We have $$\omega^{\sharp}(P(X_{\beta}))|_K=\omega^{\sharp}(X_{\beta})|_K=\beta|_K,$$
where in the first equality we used that $\omega(v,-)$ vanishes on $K$ for all $v\in TC$.
This proves the desired commutativity.
\end{proof}

The following proposition is a special instance of Lemma 6.7 in \cite[Subsection 6.3.]{OhLe},
where the more general case of locally conformal symplectic manifolds is treated. 
In its formulation we make use of
the above isomorphism in order to
view $\frac{\partial s_t}{\partial t}\vert_{t=0}\in \Gamma(E)$ as an element of $\Omega^1_\mathfrak{F}(C)$.

\begin{proposition}\label{prop:infmodsym}
Suppose that $(s_t)_{t\in [0,1]}$ is a family of coisotropic sections that starts at the zero-section
and is trivial under symplectic equivalence, i.e. there is a symplectic isotopy $\phi_t$ such that
 the image of the zero section under $\phi_t$ coincides with
 the graph of $s_t$.

Then the cohomology class of $\frac{\partial s_t}{\partial t}\vert_{t=0}$ in $H^1_\mathfrak{F}(C)$ 
lies in the image of $r: H^1(C) \to H^1_\mathfrak{F}(C)$.
\end{proposition}

\begin{proof}
Suppose that $\phi_t$ is the symplectic isotopy generated by the family of vector fields
$X_t$. Since $\phi_t$ is symplectic, $\beta_t:=i_{X_t}\omega$ is a family of closed one-forms.
By Lemma \ref{lem:techeasy}, we can write $\frac{\partial s_t}{\partial t}\vert_{t=0}$
as $P(X_0)$. By the previous lemma, this equals the image of $\beta_0$
under $r\circ \iota^*$. In particular, the cohomology class
of $\frac{\partial s_t}{\partial t}\vert_{t=0}$ coincides with the cohomology class
$(r\circ \iota^*)[\beta_0]$, hence lies in the image of $r:H^1(C) \to H^1_\mathfrak{F}(C)$.
\end{proof}

\begin{remark}\label{rem:cokernel}
Proposition \ref{prop:infmodsym} is an analogue of Proposition \ref{prop:trivialclass}, where we showed that if a family   $(s_t)_{t\in [0,1]}$ is  trivial under Hamiltonian equivalence then  the cohomology class of $\frac{\partial s_t}{\partial t}\vert_{t=0}$ is zero.
 
One can strengthen Proposition \ref{prop:infmodsym} by
observing that, by the same proof, every element 
in the image of 
the map $r \colon \Omega^1_{\mathrm{closed}}(C)\to \Omega^1_{\mathfrak{F}, \mathrm{closed}}(C)$
 is of the form $\frac{\partial s_t}{\partial t}\vert_{t=0}$,
where $(s_t)_{t\in [0,\epsilon)}$ arises
through the action of a symplectic isotopy on the zero-section.
Indeed, for every $\gamma\in \Omega^1_{\mathrm{closed}}(C)$
one considers
the symplectic isotopy generated by the vector field $(\omega^{\sharp})^{-1}(\pi^*\gamma)$.

In full analogy to  Remark \ref{TMHam}, this together with
Remark \ref{rem:Omega1cl} shows that  the formal tangent space at zero to $\Moduli^\Sym_U(C)$ is 
\begin{equation}\label{eq:coker}
\Omega^1_{\mathfrak{F}, \mathrm{closed}}/
r(\Omega^1_{\mathrm{closed}}(C))\cong H^1_\mathfrak{F}(C)/ r(H^1(C)),
\end{equation}
that is, the cokernel of
$ r: H^1(C) \to H^1_\mathfrak{F}(C)$.
The isomorphism is obtained by
 quotienting both terms on the left-hand side by $\Omega^1_{\mathfrak{F}, \mathrm{exact}}$
 and by using the following linear algebra statement for the denominator: if $f\colon V_1\to V_2$ is a linear map and $W_1,W_2$ are subspaces such that $f(W_1)=W_2$, then $f(V_1)/W_2=\mathrm{Im}([f]\colon V_1/W_1\to V_2/W_2$).
 
We note that if
$C$ is Lagrangian we have $H^1_\mathfrak{F}(C)=H^1(C)$ and $r$ is the identity, so its cokernel is trivial, as expected. 

Notice also, by the above and Remark \ref{TMHam}, that the formal tangent space at zero of $\Moduli^\Sym_U(C)$
is a quotient of the   formal tangent space to $\Moduli^\Ham_U(C)$,
and that they agree if{f} $ r: H^1(C) \to H^1_\mathfrak{F}(C)$ is the zero map. This happens for instance if $H^1(C)=0$, in which cases it is clear a priori that $\Moduli^\Sym_U(C)=\Moduli^\Ham_U(C)$, for all symplectic vector fields on $U$ are Hamiltonian. In Example 
\ref{ex:torus} below
we display an example in which $r$ is not the zero map.
\end{remark}

\subsection{The extended formal picture}\label{subsection: symplectomorphisms - extended picture}

We explain now how to interpret the equivalence relation $\sim_\Sym$ 
from the point of view of Oh and Park's $L_\infty[1]$-algebra structure on $\Gamma(\wedge E)[1]$. 

To this aim, we first need to briefly recall Voronov's derived bracket construction \cite{vor1,vor2}.

\begin{remark}[on Voronov's derived brackets]\label{rem:OPder}
\hspace{0cm}
\begin{enumerate}
\item Let $L$ be a graded Lie algebra, $\mathfrak{a}$ an abelian subalgebra and $P: L\to \mathfrak{a}$ a projection whose kernel is a Lie subalgebra.
Furthermore, suppose $X$ is a Maurer-Cartan element of $L$, i.e. $X\in L_1$ satisfying $[X,X]=0$, such that $P(X)=0$.
In \cite{vor1}, Voronov showed that the derived brackets
$$ \lambda_k(a_1\otimes \cdots \otimes a_k):= P([\cdots [[X,a_1],a_1]\cdots, a_k])$$
equip $\mathfrak{a}$ with the structure of an $L_\infty[1]$-algebra.
\item Observe that $X$ gives rise to a coboundary operator $-[X,\cdot]$ on $L$, which makes $L$ into a differential graded Lie algebra. This DGLA structure on $L$, the $L_\infty[1]$-algebra structure on $\mathfrak{a}$ described above, and additional structure maps $\lambda_{i}$ ($i\ge 1$) combine
 into an $L_\infty[1]$-algebra structure on $L[1]\oplus \mathfrak{a}$, see \cite{vor1,vor2}. The  
{additional structure maps}   take values in $\mathfrak{a}$ and are given by
$$ \lambda_{k+1}(l[1] \otimes a_1\otimes \cdots \otimes a_k) := P([\cdots [[l,a_1],a_2]\cdots,a_k]),$$
where $l\in L$ and $k\ge 0$, $a_1,\cdots ,a_k \in \mathfrak{a}$. Notice that
for $k=0$ we obtain
$\lambda_1(l[1])=P(l)$.
\end{enumerate}
\end{remark}

Since $\mathfrak{a}$ is a $L_\infty[1]$-subalgebra of $L[1]\oplus \mathfrak{a}$,
the inclusion $\beta \mapsto (0,\beta)$ identifies Maurer-Cartan elements of $\mathfrak{a}$ with those
Maurer-Cartan elements of $L[1]\oplus \mathfrak{a}$ which lie in $\{0\}\oplus \mathfrak{a}$.
We use this identification to obtain a new equivalence relation on $\mathsf{MC}(\mathfrak{a})$.
To this aim, we need to modify $L[1]\oplus \mathfrak{a}$ slighty to guarantee that the set of Maurer-Cartan elements in $\{0\}\oplus \mathfrak{a}$ is preserved
by the gauge-action:

\begin{lemma}\label{lemma: extended gauge}
Let $Z(X) \subset L$ be the graded Lie subalgebra of elements $\sigma$ which commute with $X$.
{
\begin{enumerate}
\item $Z(X)[1]\oplus \mathfrak{a} \subset L[1]\oplus \mathfrak{a}$ is an $L_\infty[1]$-subalgebra.
\item The gauge-equivalence in $Z(X)[1]\oplus \mathfrak{a}$ preserves the set of Maurer-Cartan elements in $\{0\}\oplus \mathfrak{a} \subset L[1]\oplus \mathfrak{a}$.
\end{enumerate}
}
\end{lemma}

\begin{proof}
{
The first claim reduces to the fact that $Z(X)$ is a graded Lie subalgebra of $L$.}

{
Concerning the second claim,
we consider the effect of the gauge-action on 
first component of $L[1]\oplus \mathfrak{a}$. We find}
$$ \frac{d}{dt} l_t = [X,\sigma_t] + [l_t,\sigma_t],$$
where $\sigma_t$ is a family of elements in $L_{0}$ and $l_t$ in $L_1$. Now if we require $\sigma_t$ to lie in $Z(X)$, the term $[X,\sigma_t]$ is zero and
we recover the usual adjoint action of $L_0$ on $L_1$, for which the origin is clearly a fixed point. 
\end{proof}

\begin{remark}
The restriction to Maurer-Cartan elements in $\{0\}\oplus \mathfrak{a}$ of the gauge-equivalence of $Z(X)[1]\oplus \mathfrak{a}$ can be alternatively described as follows: 
It is straight-forward to check that if $L'$ is any graded Lie subalgebra of $L$ closed under $[X,\cdot]$, then $L'[1]\oplus \mathfrak{a}$
is closed w.r.t. all the multibrackets  of the
 $L_{\infty}[1]$-algebra $L[1]\oplus \mathfrak{a}$. We apply this to
 $L'=Z_0(X)$, the degree zero component of $Z(X)$,
 to obtain an $L_{\infty}[1]$-algebra
 $Z_0(X)[1]\oplus \mathfrak{a}$.  Notice that   $\mathsf{MC}(Z_0(X)[1]\oplus \mathfrak{a})= \mathsf{MC}(\{0\}\oplus\mathfrak{a})$, simply because $Z_0(X)[1]$ is concentrated in degree $-1$ while Maurer-Cartan elements have degree zero. Hence the gauge-equivalence of $Z_0(X)[1]\oplus \mathfrak{a}$ on its Maurer-Cartan elements agrees with the the restriction of the gauge-equivalence appearing in Lemma
\ref{lemma: extended gauge}.
\end{remark}

This result prompts us to give the following definition

\begin{definition}\label{def:extended gauge}
Two Maurer-Cartan elements $\beta_0$ and $\beta_1$ of $\mathfrak{a}$ are called {\bf extended gauge-equivalent}, written $\beta_0 \sim_{\ex-\gauge}\beta_1$, if there is a one-parameter family $\sigma_{t}$ of degree $0$ elements of $L$ which commute with $X$ and a one-parameter family $\beta_{t}$ of elements of $\mathfrak{a}_0$, agreeing with the given ones at $t=0$ 
  and $t=1$, such that
\begin{equation*}
\frac{\partial}{\partial t} \beta_{t}=P(\sigma_{t})+ P([\sigma_{t},\beta_{t}] )+\frac{1}{2!}P([[\sigma_t,\beta_{t}],\beta_{t}])+\frac{1}{3!} P([[[\sigma_t,\beta_{t}],\beta_{t}],\beta_{t}])+\dots
\end{equation*}
holds for all $t\in [0,1]$.
\end{definition}

We note that in the above definition we only allow gauge-equivalences generated by elements coming from the component $L[1]$, which seems
more restrictive than considering arbitrary gauge-equivalences in $Z(X)[1]\oplus \mathfrak{a}$. However, observe
that families of elements of the form $[X,\gamma_t]$, for $\gamma_t \in \mathfrak{a}_{-1}$, automatically
commute with $X$ and hence give rise to extended gauge-equivalences.
If we substitute such a family $[X,\gamma_t]$ for $\sigma_t$ in the above formula, we obtain
\begin{equation*}
\frac{\partial}{\partial t} \beta_{t}=P([X,\gamma_{t}])+ P([[X,\gamma_{t}],\beta_{t}] )+\frac{1}{2!}P([[[X,\gamma_t],\beta_{t}],\beta_{t}])+\frac{1}{3!} P([[[[X,\gamma_t],\beta_{t}],\beta_{t}],\beta_{t}])+\dots.
\end{equation*}
This expression coincides with the defining formula of an (ordinary) gauge-equivalence between
the Maurer-Cartan elements
$\beta_0$ and $\beta_1$,
see Definition \ref{def:eqMC} in Subsection \ref{subsection: Ham diffeos - equivalences}. Hence $\sim_{\ex-\gauge}$ from
Definition \ref{def:extended gauge} really coincides with the gauge-equivalence inherited from $Z(X)[1]\oplus \mathfrak{a}$ and
we furthermore see that ordinary gauge-equivalence implies extended gauge-equivalence.

\begin{remark}
One can obtain every $L_\infty[1]$-algebra from the derived bracket construction, see \cite[Example 4.1]{vor1} and \cite[Appendix A.3]{YaelZ} for details:
Let $W$ be a graded vector space and denote its graded symmetric coalgebra by $SW:=\oplus_{i\ge 0}S^iW$, where $S^iW$ can be described as the fixed point set of the $i$-fold tensor algebra $T^iW$ on $W$ under the even action of the symmetric group $\Sigma_i$.
The deconcatenation map $\Delta: TW \to TW\otimes TW$ given by
\begin{eqnarray*}
\Delta(x_1\otimes \cdots \otimes x_n) &:=& 1\otimes (x_1\otimes \cdots \otimes x_n) +\sum_{i=1}^{n-1} (x_1\otimes \cdots \otimes x_i)\otimes (x_{i+1}\otimes \cdots \otimes x_n) +\\ && (x_1\otimes \cdots \otimes x_n) \otimes 1
\end{eqnarray*}
restricts to $SW$ and defines a cocommutative coassociative coproduct there.
As essentially observed by Stasheff in \cite{Stasheff},
an  
$L_\infty[1]$-algebra structure on $W$ is the same as a degree $1$ coderivation $D$ of the coalgebra
$SW$ that annihilates $1\in \mathbb{R}\subset SW$ and squares to zero, i.e. an endomorphism $D$ of $SW$ that satisfies
$$  \Delta \circ D = (D\otimes \mathrm{id} + \mathrm{id}\otimes D) \circ \Delta, \quad D(1)=0, \quad \textrm{and} \quad D\circ D=0.$$
This means that an $L_\infty[1]$-algebra structure on $W$ corresponds to a Maurer-Cartan element
$D$ in the graded Lie algebra of coderivations $\mathrm{Coder}(SW)$, equipped with the commutator bracket.

One can reinterpret this construction in terms of the higher derived bracket construction as follows:
For $L$ we take $\mathrm{Coder}(SW)$ and as the abelian subalgebra we take $W$, which sits inside
$\mathrm{Coder}(SW)\cong \mathrm{Hom}(SW,W)$ as those homomorphisms which map $1$ to an element of $W$ and everything else to $0$.
The projection map $P: \mathrm{Coder}(SW)\cong \mathrm{Hom}(SW,W) \to W$ is evaluation at $1 \in \mathbb{R} = S^0W$ and the Maurer-Cartan element $X$ is the coderivation $D$.
The corresponding derived brackets just return the $L_\infty[1]$-algebra structure on $W$.

To see what extended gauge-equivalence means in this case, let $\sigma_t \in \mathrm{Coder}(SW)$
be a family of coderivation of degree $0$ which commutes with $D$.
The extended gauge-action on Maurer-Cartan elements $\beta_t$ of $W$ reads
$$ \frac{d}{dt}\beta_t = \mathrm{pr}_W(\sigma_t + \sigma_t(\beta_t)+\frac{1}{2}\sigma_t(\beta_t\otimes \beta_t)+\cdots),$$
where $\mathrm{pr}_W$ denotes the projection $SW\to W$.

Suppose we can integrate this family of coderivations to a family of automorphisms $\Phi_t$ of the coalgebra $SW$. 
By construction, $\Phi_t$ will commute with $D$ as well
and act on Maurer-Cartan elements of $W$ by 
$$ \mathrm{pr}_W(\Phi_t(1 + \beta + \frac{1}{2}\beta\otimes \beta + \frac{1}{3!}\beta\otimes \beta\otimes \beta +\cdots)).$$
This formula can by verified by checking that
differentiation yields the formula for the extended gauge-action from above.

In short, extended gauge-equivalence in the case at hand amounts to the action of those automorphisms of the $L_\infty[1]$-algebra structure $D$
which are connected to the identity.
\end{remark}

We now return to the equivalence relation $\sim_\Sym$ on the space of coisotropic deformations.
If one applies Voronov's derived bracket construction (see Remark \ref{rem:OPder})
to the data\begin{itemize}
\item $L=(\chi^\bullet(E)[1],[-,-])$,
\item $\mathfrak{a}=\Gamma(\wedge E)[1]$
\item $P: L\to \mathfrak{a}$ the projection as before,
\item $X=\Pi \in \chi^2(E)$ the Poisson bivector field corresponding to $\omega$,
\end{itemize}
 one
recovers Oh and Park's $L_\infty[1]$-algebra structure on $\Gamma(\wedge E)[1]$  from Subsection \ref{subsection: coisotropic submanifolds - L_infty}.
 
By Lemma \ref{lemma: extended gauge}, its Maurer-Cartan elements are endowed with a second equivalence relation, arising from
the degree $0$-elements
of $\chi^\bullet(E)[1]$ that commute with the Poisson bivector field.
These are exactly the symplectic vector fields.
Lemma \ref{lemma: extended gauge} prompts us to repeat the definition of gauge-equivalence from
Subsection \ref{subsection: Ham diffeos - equivalences}, with the Hamiltonian vector fields
$X_{\pi^*f_t}$ replaced with any family of symplectic vector fields. However, in order
to maintain the link to geometry, we restrict ourselves to symplectic vector fields on $E$ which are firbre-wise entrie.  

 \begin{definition}\label{definition: extended gauge}
 Let $(U,\Omega)$ be a local symplectic model for the coisotropic submanifold $C$.
 
 Two elements $s_0$ and $s_1$ of $\Defor_U(C)$ are {\bf extended gauge-equivalent},
 $s_0 \sim_{\ex-\gauge} s_1$,
 if there is a one-parameter family $s_t \in \Gamma(U)$, agreeing with $s_0$ and $s_1$ at
 the endpoints, and a family of symplectic, firbre-wise entrie vector fields $X_t$ on $U$ such that
 $$ \frac{\partial}{\partial t} (-s_t) = P(e^{[\cdot,-s_t]}X_t)$$
 holds for all $t\in [0,1]$.
\end{definition}

\begin{remark}
We denote the induced equivalence relation on $\Defor_U(C)$ by $\sim_{\ex-\gauge}$.
The proof of Proposition \ref{algequiv} goes through mutatis mutandis and we obtain:
\end{remark}

\begin{proposition}\label{algequiv_sym}
Elements $s_0$ and $s_1$ of $\Defor_U(C)$ are extended gauge-equivalent if and only if
 there is
  a one-parameter family $s_t\in \Gamma(U)$, agreeing with $s_0$ and $s_1$ at the endpoints, 
  and a one-parameter family $X_t$ of symplectic and firbre-wise entrie vector fields on $U$ such that
\begin{equation*}
\frac{\partial}{\partial t}s_t=\pr_{s_t}(X_t\vert_{\graph(s_t)})
\end{equation*}
holds for all $t\in [0,1]$.  \end{proposition}

\subsection{Symplectic equivalence $=$ extended gauge-equivalence} \label{subsection: symplectomorphisms - equivalences coincide}
Our aim is to compare the two equivalence relations $\sim_{\ex-\gauge}$ and $\sim_\Sym$ on $\Defor_U(E)$.

\begin{remark}
 The following two results are proved in parallel to Proposition \ref{prop:implications1} and Proposition \ref{prop:compacteq}.
 The key point is the following: if we are given a section $s$ of $U$ whose graph is coisotropic, and a closed $1$-form $\beta$ on $E$,
 the vector fields 
$(\omega^{\sharp})^{-1}(\pi^*s^*\beta)$ 
 and $(\omega^{\sharp})^{-1}(\beta)$
 have the same vertical projection onto $E\vert_{\graph(s)}$ along $T\graph(s)$.
 As in the proofs of Proposition \ref{prop:implications1} and Propositions \ref{prop:compacteq}, this fact allows
 one to replace
 any family of symplectic isotopies by a family of symplectic isotopies generated by firbre-wise entrie
 symplectic vector fields.
\end{remark}

\begin{proposition}\label{prop:implications1 - sym}
The following   implication holds between the  equivalence 
relations on $\Defor_U(C)$:
$$
\xymatrix{
 \textrm{symplectic equivalence } \sim_{\Sym}\ar@{=>}[d] \\
\textrm{extended gauge-equivalence } \sim_{\ex-\gauge}.}
$$
\end{proposition}

Under the assumption that $C$ is compact, we can reverse the implications of Proposition \ref{prop:implications1 - sym}:

\begin{proposition}\label{prop:compacteq - sym}
Suppose $C$ is compact coisotropic submanifold. Then the following implication holds for the local symplectic model
of $C$:
$$
\xymatrix{
\textrm{extended gauge-equivalence } \sim_{\ex-\gauge} \ar@{=>}[d]\\
\textrm{symplectic   equivalence }\sim_{\Sym}.}$$
\end{proposition}

Combining the two previous propositions, we obtain the main result of this section:
 
\begin{theorem}\label{cor:3agree - sym} 
Let $C$ be a compact coisotropic submanifold with local symplectic model $(U,\Omega)$.
The equivalence relations on
$$ \Defor_U(C) := \{s\in \Gamma(E): s \textrm{ is coisotropic and } \graph(s) \subset U\}$$
given by
\begin{itemize}
\item symplectic equivalence $\sim_{\Sym}$ (Definition \ref{def:symeq})
 and
\item extended gauge-equivalence $\sim_{\ex-\gauge}$
(Definition \ref{definition: extended gauge}, see also Proposition \ref{algequiv_sym})
\end{itemize}
coincide.
\end{theorem}

As a consequence we have:

\begin{theorem} \label{theorem: symplectic equivalence}
Let $C$ be a compact coisotropic submanifold with local symplectic model $(U,\Omega)$.
The  bijection 
$$\Defor_U(C) \cong\mathsf{MC}_U(\Gamma(\wedge E)[1]) $$
descends to a bijection
$$ \Moduli_U^{\Sym}(C) :=\Defor_U(C)/\sim_\Sym \cong \mathsf{MC}_U(\Gamma(\wedge E)[1])/\sim_{\ex-\gauge}.$$
\end{theorem}

\subsection{Comparison with Hamiltonian equivalence}\label{subsection: symplectomorphisms - comparison}

In this note we considered both 
Hamiltonian equivalence (Definition \ref{def:hamsymeq}) and symplectic equivalence (Definition \ref{def:symeq}) of coisotropic submanifolds. Here we summarize some results of Ruan \cite{Ruan} about the relation between these two kinds of equivalence.
Ruan considers a restricted class of coisotropic submanifolds, which he calls integral.

\begin{definition}
 A coisotropic submanifold $C$ is {\bf integral} if the leaves of its characteristic foliation $\mathfrak{F}$ are all compact and the set of leaves $S$ 
admits a smooth structure such that the natural map {$C\to S$}
 is a submersion. (In other words: $C\to S$ is a smooth fibre bundle with compact fibres.)
\end{definition}

\begin{remark}
\hspace{0cm}
\begin{enumerate}
 \item As Ruan noticed in \cite{Ruan}, being integral is not preserved under small deformations inside
 the space of coisotropic submanifolds. In the following,
 we restrict attention to the space of coisotropic sections which are integral, and denote them by
 $\Defor_U^{\mathrm{int}}(C)$.
 \item Recall that every fibre bundle $p: C\to S$ with compact fibres $S$ inherits a local system ${\bf H}$, given by the fibrewise cohomology,
 i.e. $$ {\bf H}_s := H^\bullet(p^{-1}(s),\mathbb{R}),$$
 equipped with the Gauss-Manin connection.
 The cohomology $H^\bullet(S,{\bf H})$ is the second sheet of the Leray-Serre spectral sequence associated to
 $p: C\to S$, which converges to the cohomology of $C$.
 We will focus on ${\bf H}^1_s:=H^1(p^{-1}(s),\mathbb{R})$.
 Observe that the differential $d_2$ of the second sheet gives a natural linear map
 $$ d_2: H^0(S,{\bf H}^1) \to H^2(S,{\bf H}^0).$$
Notice that the former group is  the space of global, flat sections of the vector bundle ${\bf H}^1$ over $S$.
 Since the fibres of $p$ are connected, the latter group is just $H^2(S,\mathbb{R})$. 
\end{enumerate}
\end{remark}

In \cite[Theorem 1]{Ruan} Ruan establishes the following result:

\begin{theorem}\label{thm:ruan}
 Let $C$ be an integral coisotropic submanifold.
 
 \begin{enumerate}
  \item There is an open embedding    $$ \Defor_U^{\mathrm{int}}(C) / \sim_\Ham \hookrightarrow H^0(S,{\bf H}^1).$$
  \item
  The image of the equivalence class of $C$ with respect to symplectic equivalence
  $\sim_{\Sym}$ under the map $\Defor_U^{\mathrm{int}}(C) \to \Defor_U^{\mathrm{int}}(C)/\sim_\Ham$ is given nearby $C$
  as the kernel of
  $d_2: H^0(S,{\bf H}^1)\to H^2(S,\RR)$.
 \end{enumerate}

\end{theorem}

Below we reproduce an example from \cite{Ruan}:
\begin{example}
Consider the unit sphere $C=S^3$ in $\RR^4$, with the canonical symplectic form. The characteristic leaves of $S^3$ are circles, and $p\colon S^3\to S=S^2$
is the Hopf fibration. ${\bf H}^1$ is a trivial rank 1 vector bundle over $S^2$, so $H^0(S,{\bf H}^1)\cong \RR$, one generator being represented by a connection 1-form on the Hopf fibration. The map $H^0(S,{\bf H}^1)\to H^2(S,\RR)\cong \RR$ is an isomorphism, reflecting the fact that the connection is not flat.

Hence, by Theorem \ref{thm:ruan}, not all nearby integral coisotropic deformations of $S^3$ are related to $C$ by 
a symplectomorphism, for instance all spheres of radius $r$ for $r\neq 1$ are not. But those which are, are actually equivalent to $C$ by a Hamiltonian diffeomorphism. The latter statement follows, since  $H^1(C)=0$ implies that all symplectic vector fields in a tubular neighborhood of $C$ are Hamiltonian.
\end{example}

Another example is:

\begin{example}\label{ex:torus}
Consider the 3-torus $C=\TT^3$, which ``coordinates'' $\theta_1,\theta_2,\theta_3$, as the zero section of $(\TT^3 \times \RR,
d\theta_1\wedge d\theta_2+d\theta_3\wedge dx_4)$, where $x_4$ is the standard coordinate on $\RR$. The characteristic leaves are again circles, and $p\colon \TT^3\to S=\TT^2$
is the trivial fibration. Again, ${\bf H}^1$ is a trivial rank 1 vector bundle, so $H^0(S,{\bf H}^1)\cong \RR$, one generator being represented by
$d\theta_3$. The map $H^0(S,{\bf H}^1)\cong \RR\to H^2(S,\RR)\cong \RR$ is the zero map, reflecting the fact that  $d\theta_3$ is a closed 1-form.

We conclude that all nearby integral coisotropic deformations of $C$ are related to $C$ by 
a symplectomorphism, but not all of them are related to $C$ by a Hamiltonian diffeomorphism. For instance, the 3-tori given by $\{x_4=c\}$ for constants $c\neq 0$ are not.
Notice that the latter statement is in accordance with the fact that 
 $\Moduli^\Sym_U(C)\neq \Moduli^\Ham_U(C)$, which is a consequence of 
 Remark \ref{rem:cokernel} since the map
  $ r: H^1(C) \to H^1_\mathfrak{F}(C)$ has one-dimensional image.

 \end{example}

\section{The transversally integrable case}\label{section: transversal}

In this section we consider coisotropic submanifolds $C$ that admit a foliation
that is complementary to the characteristic foliation:

\begin{definition}
 A coisotropic submanifold $C$ of $(M,\omega)$ is called {\bf transversally integrable} if
 the kernel $K$ of the pre-symplectic structure $\omega_C$ admits a complementary subbundle $G$
 which is involutive.
\end{definition}

\begin{remark}
A transversally integrable coisotropic submanifold $C$ comes equipped with two foliations:
the characteristic foliation $\mathfrak{F}$, given by the maximal leaves of $K$,
and another foliation, given by the maximal leaves of $G$.
Since $K$ is the kernel of the pre-symplectic structure on $C$,
the leaves of $G \cong TC/K$ inherit a symplectic structure.
\end{remark}

The assumption of transversal integrability leads to many simplifications. We recover a result by Oh and Park \cite{OP} that says that the $L_\infty[1]$-algebra associated to a transversally integrable $C$ is a differential graded Lie algebra (Proposition \ref{ref:propDGLAOP}). Moreover, we give
a formula for the coisotropic section generated by moving the zero section by a basic Hamiltonian flow (Proposition \ref{prop:section}).

\subsection{Oh and Park's $L_\infty[1]$-algebra} \label{subsection: transversal - L_infty}

Let $C$ be a coisotropic submanifold and $(U,\Omega)$ be the local symplectic model
of $C$ as in Section \ref{section: pre-symplectic geometry}. As seen there, the normal model is a neighborhood of the zero section in a vector bundle $E \to C$, so it
comes equipped with a surjective submersion $\pi \colon U\to C$.

The following proposition was already proven in \cite[Equation (9.17)]{OP} (see also Theorem 9.3 there). We provide an alternative proof here.

\begin{proposition}\label{ref:propDGLAOP}
Let  $C$ be a coisotropic submanifold, and assume there exists an involutive complement $G$ to $K=\ker(\omega_C)$. 
Then the $L_{\infty}[1]$-algebra structure on $\Gamma(\wedge E)[1]$, $E=K^*$,
associated to $C$ as in Subsection \ref{subsection: coisotropic submanifolds - L_infty}
corresponds\footnote{That is, the $L_{\infty}$-algebra obtained after applying the degree shift operator $[-1]$ is a differential graded Lie algebra, i.e. the structure maps $\lambda_k$ vanish for $k>2$.} to a differential graded Lie algebra.
\end{proposition}
 
\begin{proof}

The structure maps $\lambda_r$ of the $L_\infty[1]$-algebra
from Subsection \ref{subsection: coisotropic submanifolds - L_infty} are derivations in each argument. Consequently
they can be evaluated locally. Moreover, the derivation property and a degree-count using the fact that $\Pi$ is a bivector field show that it suffices to evaluate them on tuples of the form
$$(f,g,s_1,\dots,s_{r-2}), \quad (f,s_1,\dots,s_{r-1}) \quad  \textrm{and}  \quad(s_1,\dots,s_r)$$
with $f, g \in \mathcal{C}^\infty(C)$ and $s_i \in \Gamma(E)$, seen as vertical vector fields on $E$, in order to determine them completely.
   
We now compute the multibracket $\lambda_k$
of Oh and Park's $L_\infty[1]$-algebra structure on $\Gamma(\wedge E)[1]$
in local coordinates. 
As we already noticed, the leaves of the
involutive subbundle $G$ complementary to $K$
   are symplectic.
Choose coordinates $q_1,\dots,q_{n-k},y_1,\dots,y_{2k}$
on 
$C$ adapted to the foliations integrating $K$ and $G$, respectively. That is, $K$ is spanned by the $\pd{q}$'s and
$G$ is spanned by the $\pd{y}$'s. 
Add conjugate coordinates $p_1,\dots,p_{n-k},u_1,\dots,u_{2k}$ to obtain a coordinate system on $T^*C$.
The subbundle $G^{\circ}\subset T^*C$ is locally given by $\{u_1=\dots=u_{2k}=0\}$. Hence the symplectic form on $E=K^*\cong G^{\circ}$ (see Section \ref{section: pre-symplectic geometry}) reads 
\begin{equation}\label{eq:Omegacoord}
\Omega=\sum_{i=1}^{n-k} dq_i\wedge dp_i+\pi^*\omega_C.
\end{equation}
 Notice that, in coordinates, $\omega_C$ has the form $\sum h_{jl}dy_j\wedge dy_l$ for some functions $h_{jl}$ on $C$.
Notice further that $\Omega$ (and therefore the  Poisson bivector field $\Pi$ obtained by inverting $\Omega$) are invariant under all of the vertical vector fields  
 $\frac{\partial}{\partial p_1},\dots,\frac{\partial}{\partial p_{n-k}}$.
 The structure maps $\lambda_r$ are determined by their evaluation on tuples of the form
 $$ (f, g , \frac{\partial}{\partial q_{i_1}}, \dots,  \frac{\partial}{\partial q_{i_{r-2}}}), \quad
 (f , \frac{\partial}{\partial q_{i_1}}, \dots,  \frac{\partial}{\partial q_{i_{r-1}}}) \quad \textrm{and}
 \quad (\frac{\partial}{\partial q_{i_1}}, \dots,  \frac{\partial}{\partial q_{i_{r}}}).$$

Consider the term on the right-hand side of Equation \eqref{eq:multibrackets} in Subsection \ref{subsection: coisotropic submanifolds - L_infty} before applying the projection $P$, that is,
\begin{equation*}\label{eq:pis1}
[[\dots[\Pi,-],-]\dots],-].
\end{equation*}
As we argued above, it suffices to evaluate this expression on tuples consisting of functions on $C$ and vertical vector fields
$\frac{\partial}{\partial q_{i}}$.
Since $\pi^*f$ and $\Omega$ is invariant under any of the vertical vector fields, $\frac{\partial}{\partial q_{i}}$,
the structure map $\lambda_r$ vanish whenever we evaluate it on a tuple that contains a
$\frac{\partial}{\partial q_{i}}$. Hence, only $\lambda_1$ and $\lambda_2$ can be non-zero.  
\end{proof}
   
\begin{remark}
\hspace{0cm}
\begin{enumerate}
\item
 The non-trivial structure maps of the differential graded Lie algebra associated to a transversally integrable
 coisotropic submanifold are given by
 $$ \lambda_1(f) =  P(X_{\pi^*f}) \quad \textrm{and} \quad \lambda_2(f,g) = - \{f,g\}^G,$$
the fact that $\lambda_1$ and $\lambda_2$ annihilate the coordinate vector fields $\pd{q}$ associated to adapted coordinates on $C$,
 and the derivation rule. Here, $\{\cdot,\cdot\}^G$ denotes the leafwise Poisson structure
 associated to the symplectic foliation integrating $G$.\footnote{The additional minus sign in $\lambda_2$ is a consequence of the fact that we work in $\Gamma(\wedge E)[1]$, i.e. that we shift all the degrees down be one. In particular, functions have degree $-1$ after this shift.}
\item {The $L_\infty[1]$-algebra we associated to a coisotropic submanifold $C$ depends on the choice
of a tubular neighborhood $U$, as well as on the choice of a subbundle $G$ complementary to the kernel $K$ of the pre-symplectic structure.
Theorem 4.3 of \cite{CS} asserts that different choices of these data lead to {\em isomorphic} $L_\infty[1]$-algebras. Consequently
Proposition \ref{ref:propDGLAOP} guarantees that in case an involutive transversal distribution exists, every $L_\infty[1]$-algebra
associated to $C$ is isomorphic to a differential graded Lie algebra.}
\end{enumerate}
\end{remark}

\subsection{Hamiltonian equivalences}\label{subsection: transversal - Hamiltonian}

We want to be more explicit about lifting constructions from a coisotropic submanifold $C$
to its local symplectic model $(U,\Omega)$. To this end, the concept of a partial Ehresmann connection will be of great
importance.

\begin{definition}
Let $K$ be an involutive distribution on $C$. Suppose $\pi: U\to C$ is a surjective submersion.
A {\bf partial Ehresmann connection} on $U$ is a choice of a complementary subbundle $G$ to $K$
and a subbundle $G^\sharp$ of $TU$ such that
the differential of $d_x\pi$ at $x\in U$ maps $G^{\sharp}_x$ isomorphically onto $G_{\pi(x)}$.
\end{definition}

\begin{remark}
We notice that the last condition implies that $G^\sharp$ is complementary to $(d\pi)^{-1}(K)$.
\end{remark}

\begin{lemma}
Let $C$ be a coisotropic submanifold with local symplectic model $(U,\Omega)$.
Suppose $G$ is the subbundle complementary to the kernel $K$ of the pre-symplectic form $\omega_C$
which was chosen in the construction of $(U,\Omega)$.
\begin{enumerate}
 \item The subbundles $(d\pi)^{-1}G$ and $V=\ker(d\pi)$ of $TU$ are symplectically
 orthogonal to each other.
\item The bundle $$G^\sharp := ((d\pi)^{-1}(K))^\perp$$
defines a partial Ehresmann connection on $U$.
\end{enumerate}
\end{lemma}

\begin{proof}
We take $\xi\in \ker(d_x\pi)$ and $v\in (d_x\pi)^{-1}(G)$.
Plugging the two vectors into the symplectic form $\omega$ yields
$$ \Omega_x(\xi,v) = \omega_C(d\pi(\xi),d\pi(v)) + \omega_{T^*C}(d_xj(\xi),d_xj(v)) = 0.$$
Since the ranks of the two subbundles add up to the rank of $TU$, the first claim follows.

Concerning (2), the inclusion $\ker(d_x\pi) \subset (d_x\pi)^{-1}(K)$
implies
$$ ((d_x\pi)^{-1}(K))^\perp \subset (\ker(d_x\pi))^\perp = (d_x\pi)^{-1}(G),$$
i.e. $G^\sharp$ maps indeed onto $G$ under $d\pi$.
To check that the map is an isomorphism, it suffices to check that the dimensions match, which is straight-forward.
\end{proof}

\begin{remark}
The partial Ehresmann connection $G^\sharp$ was first considered in \cite{OP}, see Equation (6.3) there.
Observe that $G^\sharp$ is usually not linear, i.e. not compatible with the linear structure
on $E \supset U$.

A partial Ehresmann connection $G^\sharp$ is called flat if it is an involutive subbundle of $TU$.
This condition can be restated as follows: $G^\sharp$ gives rise to a map
$$ \Gamma(G) \to \mathcal{X}(U), \quad X\mapsto X^{\mathrm{hor}},$$
where $X^{\mathrm{hor}}$ is uniquely determined by the condition
$d_x\pi(X^\mathrm{hor}\vert_x)=X_{\pi(x)}$ for all $x\in U$.
Flatness of $G^\sharp$ is equivalent to the requirements that $G$ is involutive and that the map
$X\mapsto X^\mathrm{hor}$ is compatible with the Lie bracket of vector fields.
\end{remark}

\begin{proposition}\label{lem:flatlift}
Let $C$ be a coisotropic submanifold that is transversally integrable, with $G$ an involutive transversal distribution. Let $(U,\Omega)$ be the corresponding local symplectic model.

\begin{enumerate}
 \item The partial connection $G^{\sharp}$ on $U\subset E$ is linear
and flat.
\item For all $f\in \cC^{\infty}(C)$, we have
$$X_{\pi^*f}=P(X_{\pi^*f})+(X_f^G)^{\mathrm{hor}}$$
where: 
\begin{itemize}
 \item[(i)] $P(X_{\pi^*f}) \in \Gamma(E)$ is seen as a vertical vector field on $U\subset E$, constant along the fibres,
 \item[(ii)] $X_f^G$ denotes the leafwise Hamiltonian vector field of $f$ with respect to the symplectic foliation integrating $G$ and
 \item[(iii)] $(X_f^G)^{\mathrm{hor}}$ denotes the horizontal lift of $X_f^G$ with respect to the partial Ehresmann connection $G^{\sharp}$.
\end{itemize}

\end{enumerate}

\end{proposition}

\begin{proof}
Choose coordinates $y_1,\dots,y_{2k},q_1,\dots,q_{n-k},p_1,\dots,p_{n-k}$ on $U$ as in the proof of Proposition \ref{ref:propDGLAOP}. 

(1) Equation \eqref{eq:Omegacoord} shows that at every point $x\in U$, $\Omega_{x}$ is the sum of two symplectic forms, one defined on the subspace spanned by the $\pd{p}$ and $\pd{q}$'s, the other one defined on the subspace spanned by the $\pd{y}$'s. 
As $((d\pi)^{-1}E)^{\circ}$ is spanned by the $dy$'s, we obtain 
\begin{equation}\label{eq:Gsharp}
G^{\sharp}=\mathrm{span}\{\pd{y_1},\dots,\pd{y_{2k}}\}.
\end{equation}
In other words, in the trivialization of the vector bundle $E=K^*$ given by the chosen coordinates, $G^{\sharp}$ is a trivial partial connection. 

From this we deduce that the parallel transport with respect to $G^{\sharp}$ along paths contained in a leaf of $G$ is given by linear isomorphisms between the fibres of $E$, showing that the partial connection  $G^{\sharp}$ is linear. Second, the linear partial connection  $G^{\sharp}$ is flat, since the distribution $G^{\sharp}$  is clearly involutive.

(2)
In the above coordinates, by Equation \eqref{eq:Omegacoord}, 
we have
\begin{equation*}\label{eq:xfcoord}
X_{\pi^*f}=[\Pi,\pi^*f]=\sum_i\frac{\partial f}{\partial q_i}
\frac{\partial}{\partial p_i}+
(X_f^G)^{\mathrm{hor}},
\end{equation*}
where  for the horizontal component we used \eqref{eq:Gsharp}.
Its vertical component is invariant under each of the vertical vector fields  
 $\frac{\partial}{\partial p_1},\dots,\frac{\partial}{\partial p_{n-k}}$, therefore it agrees with the vertical component at $\pi(x)\in C$, which is $P(X_{\pi^*f})$. \end{proof}
 
\begin{remark}
Our next aim is to explicitly describe the sections of $E$ which are Hamiltonian equivalent to the zero section $\iota: C\to E$. 
If $C$ is compact, we can replace Hamiltonian equivalence by base Hamiltonian equivalence, see
Definition \ref{def:baseham} and Propositions \ref{prop:implications1} and \ref{prop:compacteq}.
Recall that this means that we have to consider the time one flow of a time-dependent
vector field $X_{\pi^*f_{t}}$ where $f_{t}\in \cC^{\infty}(C)$. 
Such vector fields are not vertical in general, hence solving explicitly the ODE to find their flow is not easy. We are able to do so when $G$ is involutive, making use of the following result:
\end{remark}

\begin{lemma}\label{lem:nabla} Let $A\to M$ be a vector bundle with a linear connection $\nabla$. Let $(X_t)_{t\in [0,1]}$ be a one-parameter family of vector fields on $M$, and $(\alpha_t)_{t\in [0,1]}$ a one-parameter family of sections of $A$. Consider  the one-parameter family of vector fields on $A$ given by
$$\alpha_{t}+(X_t)^{\mathrm{hor}}$$
where $\alpha_t$ is viewed as a vertical vector field which is constant along the fibres of $A$, and $(X_t)^{\mathrm{hor}}$ is the horizontal lift of $X_t$ with respect to the connection $\nabla$.
The integral curve of $\alpha_t+(X_t)^{\mathrm{hor}}$
starting at $q\in C$ is given by 
  $$s(t)=\int_0^t \,_{\gamma(\tau)}^{\gamma(t)}\para \, [\alpha_\tau|_{\gamma(\tau)}]d\tau\in A_{\gamma(t)},$$
where   
$\para$ denotes the parallel transport with respect to $\nabla$ along the curve 
$\gamma(t):=\psi_t(q)$, and $\psi_t: C\to C$  the flow of $X_t$.
\end{lemma}
\begin{proof}
We have $s=A\circ \Delta$, where $\Delta\colon [0,1]\to [0,1]^2$ is the diagonal map and 
$$A(r,t):= \int_0^{r} \,_{\gamma(\tau)}^{\gamma(t)}\para \, [\alpha_\tau|_{\gamma(\tau)}]d\tau.$$
Hence  $\frac{\partial}{\partial t}|_{t_0}s(t)=\frac{\partial}{\partial t}|_{t_0}A(t,t_0)+
\frac{\partial}{\partial t}|_{t_0}A(t_0,t)$ can be written as the sum of two terms, for all $t_0\in [0,1]$. The first one is the vertical vector 
$$\frac{\partial}{\partial t}|_{t_0}\int_0^t \,_{\gamma(\tau)}^{\gamma(t_0)}\para \, [\alpha_\tau|_{\gamma(\tau)}]d\tau=\alpha_{t_0}|_{\gamma(t_0)},$$
as can be seen noticing that the integrand is a curve in 
$A_{\gamma(t_0)}$, parametrized by $s$, and applying the fundamental theorem of calculus.

For the second term, we claim that 
\begin{equation}\label{eq:secondddr}
\frac{\partial}{\partial t}|_{t_0}\int_0^{t_0} \,_{\gamma(\tau)}^{\gamma(t)}\para \, [\alpha_\tau|_{\gamma(\tau)}]d\tau=
(X_{t_0})^{\mathrm{hor}}|_{s(t_0)}.
\end{equation}
 Notice that the integral on the left-hand side of Equation \eqref{eq:secondddr} is an element of the fibre of $V$ over $\gamma(t)$, hence applying $\frac{\partial}{\partial t}|_{t_0}$ we obtain an element of $T_{s(t_0)}A$ that projects to $\frac{\partial}{\partial t}|_{t_0}\gamma(t)=X_{t_0}|_{\gamma(t_0)}$ under $\pi$.
 We now argue that the left-hand side of Equation (\ref{eq:secondddr}) is a horizontal lift, which would
 conclude the statement.
Let $r,t\in [0,1]$. Under the identification $A_{\gamma(r)}\cong A_{\gamma(t)}$ given by the parallel transport $\,_{\gamma(r)}^{\gamma(t)}\para$, the elements $\,_{\gamma(\tau)}^{\gamma(r)}\para \, [\alpha_\tau|_{\gamma(\tau)}]$ and $\,_{\gamma(\tau)}^{\gamma(t)}\para \, [\alpha_\tau|_{\gamma(\tau)}]$ agree for every $\tau$. The same holds for the integral from $\tau=0$ to $\tau=t_0$ of these elements, 
since parallel transport is a linear isomorphism. Hence the integral on the left-hand side of \eqref{eq:secondddr}, as $t$ varies, defines a parallel section of $A$ over $\gamma$. Therefore, applying $\frac{\partial}{\partial t}|_{t_0}$ to it yields
 an horizontal element of $T_{s(t_0)}A$.
 \end{proof}

\begin{proposition}\label{prop:section} 
Let $C$ be a compact coisotropic submanifold that is transversally integrable, with $G$ an involutive transversal distribution. Let $(U,\Omega)$ be the corresponding local symplectic model.
Take a one-parameter family $(f_{t})_{t\in [0,1]}\in \cC^{\infty}(C)$,
and denote by $\Phi$ the time-$1$ flow of the time-dependent vector field $(X_{\pi^*f_t})_{t\in [0,1]}$. Then $\Phi(C)$ is the graph of the following section of $U\subset E$:

 $$ p\mapsto \int_0^1 \,_{\sigma(t)}^{\sigma(1)}\para \, [P(X_{\pi^*f_t})|_{\sigma(t)}]dt$$
where
$\para$ denotes the parallel transport with respect to the partial connection $G^{\sharp}$ along the curve 
$\sigma(t):=\psi_t((\psi_1)^{-1}p)$, for 
$\psi_t\colon C\to C$   the flow of $X^G_{f_t}$.
\end{proposition}

\begin{proof}
By Proposition \ref{lem:flatlift}, $G^{\sharp}$
is a 
 partial linear connection  on $U\subset E$, and $X_{\pi^*{f_t}}=P(X_{\pi^*f_t})+(X_{f_t}^G)^{\mathrm{hor}}$.
 We note that, in particular,
this vector field covers   
$X^G_{f_t}$, which is tangent to the leaves of $G$.
Fix $p\in C$, and let $L\subset C$ be the leaf of $G$ through $p$.
Consider the vector bundle $E\vert_{L}\to L$, equipped with the linear connection obtained
 by restricting $G^{\sharp}$.
 We apply Lemma \ref{lem:nabla} to the 
 one-parameter family of vector fields $(X^G_{f_t})|_L$ and to the one-parameter family of sections $P(X_{\pi^*f_t})|_L$.
Choosing the point $q$ so that $\psi_1(q)=p$ and setting
 $t=1$ finishes the proof.
\end{proof}

\begin{remark}
\hspace{0cm}
\begin{enumerate}
 \item We observe that Propositions \ref{lem:flatlift} and \ref{prop:section} continue to hold for symplectomorphisms,
i.e. one obtains explicit formulae for the symplectic vector field associated to a closed $1$-form obtained via
pull-back from the base $C$, as well as for the image of $C$ under the flow of such a vector field.
\item When $C$ is Lagrangian,  $U$ is open in the cotangent bundle $T^*C$, hence $X_{\pi^*f_t}$ is a (constant) vertical vector field and $P(X_{\pi^*f_t})=df_t$. 
 Further $G=\{0\}$, so the curve $\sigma$ through $p$ is constant. Therefore we recover the well-known result that $\Phi(C)$ is the graph of the exact one-form $d(\int_0^1 f_t dt)$.
\end{enumerate}

\end{remark}

 We exemplify the above discussion in the case of $C$ hypersurface, i.e. of co-dimension $1$.
 While all smooth deformations of a co-dimension $1$ submanifold are automatically coisotropic, it turns out that the equivalence problem is non-trivial.

\begin{example}
Fix a codimension $1$ compact submanifold $C$ of $(M,\omega)$, which we assume to be oriented.
The annihilator $TC^{\circ}\cong K$ is a trivial line bundle, so there is $\alpha\in \Omega^1(C)$ such that $G:=\ker(\alpha)$ satisfies $G\oplus E=TC$. As usual $K$ is the characteristic distribution of $C$, i.e., $K:=\ker(\omega_C)$.  
We \emph{assume} that $d\alpha=0$, which in particular implies that $G$ is involutive. 
By \cite[Exercise 3.36]{McSalTop} a tubular neighborhood of $C$ in $M$ is symplectomorphic to
 $$(U,\Omega):=(C\times I,  \pi^*\omega_C-du\wedge \pi^*\alpha),$$
where $I$ is an open interval containing $0$, $u$ the standard coordinate on $I$, and $\pi\colon C\times I\to C$ is the projection. In the following we denote by $\hat{\xi}$  the unique vector field on $C$ lying in $K$ such that $\alpha(\hat{\xi})=1$.

Take a  one-parameter family $f_t\in \cC^{\infty}(C)$,
and denote by $\Phi$ the time-$1$ flow of the vector field $(X_{\pi^*f_t})_{t\in[0,1]}$. Then  $\Phi(C)$ is the graph of 
$$s \colon C\to \RR, \quad s(p)=-\int_0^1 \hat{\xi}(f_{t})|_{\psi_t((\psi_1)^{-1}p)}dt$$
where   
$\psi_{\tau}\colon C\to C$ is the flow of $(X^G_{f_t})_{t\in [0,1]}$.
This follows from Prop. \ref{prop:section}, since 
$X_{\pi^*f}= -\hat{\xi}(f_{t})\pd{u}+X_f^G$ at points of $C$, and 
$G^{\sharp}$ is the trivial partial connection by Equation 
\eqref{eq:Gsharp}.
\end{example}

\appendix
\section{{Fibrewise entire Poisson structures}}\label{appendix: fibrewise entire}
 
In the body of the paper we worked with symplectic structures, but most of the results extend to fibrewise entire Poisson structures, as defined in \cite{OPPois}. More precisely, we assume the following set-up in this appendix:

 \begin{center}
\fbox{
\parbox[c]{12.6cm}{\begin{center} 
$U$ is a tubular neighborhood of the zero section in a vector bundle $E\to C$,\\
$\Pi$ is a fibrewise entire Poisson structure on $U$, such that the zero section $C$ is coisotropic.
\end{center}
}}
\end{center}

Apart from the symplectic case, an interesting example is when $E$ is the dual of a Lie algebroid $(A,\rho,[\cdot,\cdot])$ and $\Pi$ the canonical Poisson structure defined there. As described in \cite[Remark 4.5]{Ji}, one can furthermore enhance this example as follows:
given a Lie subalgebroid $B\hookrightarrow A$, its fibrewise annihilator $B^\circ \subset E$ is a coisotropic submanifold. 
If the Lie algebroid structure varies in an analytic fashion along the normal bundle to the base of $B$,
one can find a tubular neighborhood of $B^\circ \subset E$ such that $\pi$ becomes fibrewise entire.
\\

The results obtained in Sections \ref{section: coisotropic submanifolds} and \ref{section: Ham diffeos} continue to hold if one replaces the Lie algebroid $K=\mathrm{ker}\omega_C$, which is no longer defined, with $(TC)^{\circ}=E^*$, the Lie algebroid associated to the coisotropic submanifold $C$ of $(U,\Pi)$. Consequently, one has to replace the
foliated de Rham complex  $\Omega_\mathfrak{F}(C)$ with the complex $(\Gamma(\wedge E),P([\Pi,-])$.
 Many of the proofs in the main body of the article are already formulated in this setting, and some of them actually simplify in the fibrewise entire Poisson case (for instance Corrollary  \ref{prop:trivialclass}).\\

Concerning Section \ref{section: symplectomorphisms}, we 
replace ``symplectomorphisms'' in Def. \ref{def:symeq} by ``Poisson diffeomorphisms'', and denote the resulting moduli space by 
$\Moduli^\Pois_U(C)$.
The description of
    the tangent space at zero to this moduli space is now characterized in terms of Lie algebroid cohomology, as we explain in the next remark:

\begin{remark}
The tangent space at zero to $\Moduli^\Pois_U(C)$
 is isomorphic to the quotient
\begin{equation}\label{eq:cokerone}
\frac{\{s\in \Gamma(E):P([\Pi,s])=0\}}{\{P(Y): Y \text{ is a Poisson vector field on } U\}}.
\end{equation}

Indeed the numerator is the formal tangent space to $\Defor_U(C)$ by the proof of Proposition \ref{prop:EK*}. For the denominator, we argue as follows: if $Y_t$ is a one-parameter family of Poisson vector fields on $U$, and $s_t\in \Gamma(U)$ is such that the graph of $s_t$ is the image of the zero section under the time-$t$ flow of $Y_t$, then  $\frac{\partial s_t}{\partial t}\vert_{t=0}=P(Y_{0})$ by Lemma \ref{lem:techeasy}, and notice that this argument can be reversed. 

We can describe \eqref{eq:cokerone} as the cokernel of a certain map in cohomology, by finding the analog of Equation \eqref{eq:coker} that holds in the Poisson case.
We have a map $$P\colon \chi^{\bullet}(U)=\Gamma(\wedge T^*U)\to \Gamma(\wedge E)$$ between the complexes of ``forms'' for the Lie algebroid $T^*U$ on one side (the cotangent Lie algebroid of the Poisson manifold $(U,\Pi)$) and the Lie algebroid $E=(TC)^{\circ}$ on the other (the Lie algebroid of the coisotropic submanifold $C$).
The differentials are preserved, since for all $Y\in \chi^{\bullet}(U)$ we have
$P[\Pi,Y]=P[\Pi,PY]$, as a consequence of the relation $P[x,y]=P[Px,y]+P[x,Py]$ that holds in the general setting of Voronov's derived brackets. Another way to see this is to notice that $P$ is the cochain map associated to a Lie algebroid morphism, namely the inclusion of  $(TC)^{\circ}$ in $T^*U$.

Hence we obtain a map in cohomology 
$$P\colon H_{LA}(T^*U)=H_{\Pi}(U)\to H_{LA}((TC)^{\circ})$$
between the Lie algebroid cohomology of $T^*U$ (i.e. the Poisson cohomology of $(U,\Pi)$) and the Lie algebroid cohomology of $(TC)^{\circ}$. Its cokernel agrees with
  \eqref{eq:cokerone} by a linear algebra argument as in Remark \ref{rem:cokernel},
which uses the fact that for any function $F$ on $U$ we have $P[\Pi,F]=P[\Pi,F|_C]$. 
\end{remark}

Let us finally point out the place where the case of fibrewise entire Poisson structures deviates most seriously from the symplectic case:
it is no longer obvious that Poisson vector fields can be replaced by fibrewise entire ones. Therefore we cannot establish Proposition \ref{prop:implications1 - sym}, and consequently neither 
 Theorem \ref{cor:3agree - sym} nor Theorem \ref{theorem: symplectic equivalence} carry over.

\bibliographystyle{habbrv} 
\bibliography{EquivCoiso}

\end{document}